\documentclass[a4paper,10pt]{amsart}
\usepackage[foot]{amsaddr}
\usepackage[margin=2.7cm]{geometry}
\newtheorem{theorem}{Theorem}[section]
\newtheorem{lemma}[theorem]{Lemma}
\theoremstyle{remark}
\newtheorem{remark}[theorem]{Remark}
\usepackage{systeme}
\usepackage{graphicx}
\usepackage[]{units}
\usepackage{color}
\usepackage{mathrsfs}
\usepackage{bm}
\usepackage[ansinew]{inputenc}
\usepackage[breaklinks]{hyperref}
\usepackage{float}
\usepackage{amsmath}
\usepackage{isomath}
\usepackage{amsthm}
\usepackage{amssymb}

\usepackage{algorithm,algorithmic}
\usepackage{array}
\usepackage{booktabs}
\usepackage{multicol}
\usepackage{multirow}
\usepackage{tabularx}
\usepackage{subfigure}
\usepackage{cite}
\usepackage{soul}

\usepackage{hyperref} 

\theoremstyle{definition}

\newtheorem{example}[theorem]{Example}
\numberwithin{equation}{section}

\newcommand{\vertiii}[1]{{\left\vert\kern-0.25ex\left\vert\kern-0.25ex\left\vert #1 
\right\vert\kern-0.25ex\right\vert\kern-0.25ex\right\vert}}

\newcommand{\red}{\color{black}}

\begin{document}

\title[]{Random geometries {\red for} optimal control PDE problems based on fictitious domain FEMS and cut elements.
}

\author{Aikaterini Aretaki\textsuperscript{1}}
\address{\textsuperscript{1}Department of Mathematics, National Technical University of Athens, Greece}
\thanks{}
\email{kathy@mail.ntua.gr}

\author{Efthymios N. Karatzas\textsuperscript{1,}\textsuperscript{2,}\textsuperscript{3}}
\address{\textsuperscript{2}FORTH Institute of Applied and Computational Mathematics, Heraclion, Crete, Greece}
\address{\textsuperscript{3}SISSA (Affiliation), International School for Advanced Studies, Mathematics Area, mathLab Trieste, Italy}

\email{karmakis@math.ntua.gr \& efthymios.karatzas@sissa.it}




\date{\today}

\dedicatory{}

\begin{abstract}
This work investigates an elliptic optimal control problem defined on uncertain domains and discretized by a fictitious domain finite element method and cut elements. Key ingredients of the study are to manage cases  considering the usually computationally ``forbidden"  combination of poorly conditioned equation system matrices due to challenging geometries, optimal control searches with iterative methods,  
slow convergence to system solutions on deterministic and non--deterministic level, and expensive remeshing due to geometrical changes. We overcome all these difficulties, utilizing the advantages of proper preconditioners adapted to unfitted mesh methods, improved types of Monte Carlo methods, and mainly employing the advantages of embedded FEMs,  based on a fixed background mesh computed once even if geometrical changes are taking place. The sensitivity of the control problem is introduced in terms of random domains, employing a Quasi Monte Carlo method and emphasizing on a  deterministic target state. The variational discretization concept is adopted, optimal error estimates for the state, adjoint state and control are derived that confirm the efficiency of the cut finite element method in challenging geometries. The performance of a multigrid scheme especially developed for unfitted finite element discretizations adapted to the optimal control problem is also tested. Some fundamental preconditioners are applied to the arising sparse linear systems coming from the discretization of the state and adjoint state variational forms in the spatial domain.  The corresponding convergence rates along with the quality of the prescribed preconditioners are verified by numerical examples.  
\end{abstract}
\maketitle
\section{Introduction}
Embedded and immersed methods have a long history, dating back to the pioneering work of Peskin \cite{P1972}. Several improved variants can be found in the recent literature, including methods as the ghost--cell finite difference method \cite{WFC13}, cut--cell volume method \cite{PHO16}, immersed interface \cite{KB18}, ghost fluid \cite{BG14},  shifted boundary methods \cite{MS18}, $\phi$--FEM \cite{DL19}, and CutFEM \cite{BCHLM2014,BH2010,BH2012,HH2002,L19}, among others. For a comprehensive overview of this research area, the interested reader is referred to the review paper \cite{MI05}. The considerable impetus for such widespread investigations has been provided by many applications of interest, which involve general domains, e.g., in the context of fluid--structure interaction or reduced--order models for parameter--dependent domains \cite{KBR2019,KSASR2019,KSNSRa2019,KSNSRb2019} etc. In the aforementioned cases, immersed and embedded methods compare favorably to standard FEM, providing simple and efficient schemes for the numerical approximation of PDEs.

Nonetheless, the effectiveness of these methods for achieving optimal control in PDEs has not been suitably assessed. {Approaches on interface--unfitted finite elements methods for elliptic optimal control problems have been explored in \cite{WYX2019,YWX2018} and references therein. The necessity of optimal control in PDEs is ubiquitous throughout applied sciences and engineering and extensive literature on the analysis of such problems is readily available, see, for example, the classic books \cite{HPUU2008,L1971,T2010}. A class of semilinear and nonlinear elliptic optimization problems has been studied in \cite{CM2002,CT2002,Tr}, whilst some parabolic optimization problems have been surveyed in \cite{CK2014,KS2013}.}		

In many applications like optimal control for fluid flows or shape optimization, a number of physical properties, geometrical variables or material parameters are often not precisely known, which challenges input data under uncertainties, \cite{BDOS2016,BT2018,BT2019,BT,BWB2019,BLSGUR2018}. { For instance, it is important in certain cases to estimate the shape and the location of a system for the production policies to be optimized, see e.g. the boundary of an oil reservoir, the internal combustion engine and the geometry of the exhaust system which need to be optimized so as to maximize the power output of the engine. Generally, random geometries characterize materials in micro-structure and macroscopic level, mechanics of deformable bodies etc.} 
Other sources of geometric uncertainty may come from manufacturing tolerances which leads engineering analysis to a more robust design, {for instance,} the thermostat housing pipe optimal control problem, also, the distribution of gas within a pipeline may be random, since a pipeline operator, typically does not know in advance whether a power plant will come online and for how long. In all such situations, uncertainty in modeling parameters, in geometry, in initial conditions or in spatially varying material properties induces uncertainty in the outputs of the model and in any quantities of interest derived from these outputs.  For all aforementioned in this paragraph, see for example \cite{CCJ2017,CR2003,Chiu13,DDG2002,Tor2002}, \cite[pp 203-207]{tuned_exhaust77}, and references therein. 

A common way to deal with these uncertainties is to model the input data as random fields or parameter--dependent functions, { \cite{APSG2017,BOS2015,GWZ2012,GWZ2014,GM2013}}. Consequently, the derived quantities of interest are also random variables or random fields. In this case, we consider optimal control problems constrained by PDEs with random or parameterized coefficients,  \cite{GLL2011,KS2016,K1990} and/or parameterized geometries implicitly defined through random level-set functions. For an overview of fictitious domain approaches on random geometries, optimization under uncertainties and optimization under uncertainties using stochastic gradient one could be referenced to \cite{BV2019,CK2007,MKN2018}. We also refer to \cite{AUH2017} and \cite{TKXP2012}, in which a stochastic optimal control problem is elaborated with uncertainty placed in the diffusion coefficient of the elliptic PDE. The authors assume the control variable to be a random field and indicate that the stochasticity is transferred to the descent direction of the gradient method through the derivative of the cost functional, consequently the implementation of the gradient descent is more involved. On the other hand, they mention that if a deterministic control is used, the stochasticity of the problem reduces in the computation of the descent direction through the expectation of the adjoint state {\it which is built in the background mesh}. Thus the calculations are simplified due to the involvement of the expectation in the weak formulation. However, reducing the control to a deterministic variable is useful for more practical engineering applications. Although, we examine the most difficult case of non-deterministic control, as introduced in the aforementioned papers, but we place the uncertainty in the geometry of the problem.

The computational goal is to find the expected value or other statistics of the quantities generating several realizations of the random field. For each fixed realization we solve the optimal control problem numerically and we compute the quantity of interest. This procedure describes the standard Monte Carlo (MC) simulation regularly employed in such problems, \cite{F1996,LZ2004,YK1991}. 
In the present work, motivated by sensitivity analysis applied to optimal control problems on random geometries and the large amount of computer resources we need, we improve upon the slow convergence of MC simulation considering a Quasi Monte Carlo (QMC) method, which is a more effective quadrature method for tackling high dimensional stochastic integrals, \cite{AUH2017,BHP2019,DKS2013,GKKSS2019,JGTW2015,KSS2012}. 

Moreover, we investigate the performance of the cut finite element method on PDE--constrained optimization, employing the advantage of a fixed background mesh independent from the geometry variations. In this way, we avoid the mesh construction ``bottleneck" whenever the geometry changes, and we may efficiently tackle (together with a QMC method),  the iterative procedure of the optimal control search for the predescribed random data. The CutFEM method is defined in a fictitious domain setting, enforcing Nitsche--type weak boundary conditions, as well as employing an additional ghost penalty term on the interface zone element faces, \cite{BCHLM2014,BH2010,BH2012}. The presence of the latter term ensures stability and also that the resulting discretized linear system is well--conditioned, independently of the position of the physical domain concerning the fixed background mesh. We also consider a model optimal control problem on an elliptic PDE with quadratic cost functional and distributed control. The Poisson equation is examined as archetypal state constraint. Optimal error estimates for the state, control and adjoint variable are discussed. Numerical tests verify the theoretical results and illustrate the efficiency of the proposed procedure.  From the discrete optimality conditions point of view results, a numerical resolution of the linear system comes after an application of a Krylov subspace iterative method, \cite{HS2010,S2003}. 

The properties of symmetry and positive definiteness of the stiffness matrices arising from the successive discretization of the state and adjoint variational forms are inherited from the cut finite element method and enable us to use a conjugate gradient algorithm together with some fundamental preconditioners to achieve faster convergence behavior, and improved matrix conditioning. Apart from testing the classical Jacobi and the symmetrized Gauss-Seidel preconditioners \cite{LR2017,RDW2010}, a multigrid method as it is introduced in \cite{LGR2018} and adapted to unfitted mesh methods is also applied. The development of this multilevel scheme is based on a prolongation operator for unfitted finite element discretizations and also on the construction of a local correction smoother to deteriorate the effect of large jumps on the boundary zone. 

The analysis carried out in this paper is organized as follows. 
In section \ref{section2}, we formulate the distributed control problem governed by a mixed Dirichlet-Neumann boundary value Poisson problem with random input data. We derive the optimality system, which contains the state equation, the adjoint state equation, as well as, the optimality condition. Afterwards, the regularity of the state and adjoint variables is deduced from the first-order optimality conditions.
Section \ref{section3} is devoted to the discretization of the optimal control problem via the fictitious domain method with cut elements. 
In section \ref{section4}, the variational discretization approach is outlined and a--priori error estimates for the state, adjoint state and control are derived. In section \ref{section5}, we briefly describe standard MC and QMC estimators for the statistics of the random geometry optimal control problem. Section \ref{section6} outlines the basic features of the investigated preconditioners and finally, in section \ref{section7} the accuracy of the cut elements method is verified and the quality of the preconditioners is tested by some deterministic numerical examples. A  numerical example is also provided to investigate QMC convergence on a random geometry.

\section{Problem formulation}
\label{section2}
Let $(\Omega, \mathcal{F}, P)$ denote a complete probability space, where $\Omega$ is a space of samples $\omega\in\Omega$,
$\mathcal{F}$ is a $\sigma$-algebra of events and $P: \mathcal{F}\to [0, 1]$, $P(\Omega)=1$, is a probability measure.
We consider a quadratic--linear optimal control problem constrained by an uncertain geometries elliptic PDE: 
minimize the cost functional
\begin{equation}\label{cost}
\mathcal{J}(y,u):= {\frac{1}{2}\int_\mathcal{D}|y(\omega) - y_d|^2 dx + \frac{\alpha}{2}\int_\mathcal{D} |u(\omega)|^2 dx}
\end{equation}
over all $(y(\omega),u(\omega))\in H^1(\mathcal{D}(\omega))\times L^2(\mathcal{D}(\omega))$, subject to the mixed boundary value problem 
\begin{eqnarray}\label{mixedBVP}
\nonumber - \Delta y(\omega) & = & {f} + u(\omega) \quad \textrm{in}\quad \mathcal{D}(\omega), \\
y(\omega) & = & {g_D}  \quad  \,\,\,\textrm{on}\quad \Gamma_D(\omega),\\
\nonumber \mathbf{n}_\Gamma\cdot\nabla y(\omega) & = & {g_N} \quad  \,\,\,\textrm{on}\quad \Gamma_N(\omega),
\end{eqnarray}
{where $f$ is a deterministic force field and $g_D, g_N$ are deterministic boundary data; the operator $\Delta$ involves only  derivatives with respect to the spatial variable $x\in\mathcal{D}(\omega)$. To have a well-posed problem, we assume $f, g_D$ and $g_N$ to be defined on an hold-all domain, namely on the background domain (see for instance \cite{HPS2016}). Of course, the results can be extended to uncertain data $f$, $g_D$ and $g_N$.}

Here, $\mathcal{D}(\omega)\subset\mathbb{R}^2$ is a bounded domain for any outcome $\omega\in\Omega$ and ${ \bf n_{\Gamma}}$ denotes the outward pointing unit normal vector to its smooth or convex polygonal 
boundary $\Gamma(\omega)=\partial\mathcal{D}(\omega)=\Gamma_D(\omega)\cup \Gamma_N(\omega)$, decomposed in the disjoint parts $\Gamma_D(\omega)$, $\Gamma_N(\omega)$, where  Dirichlet and Neumann boundary conditions respectively are  applied.
The objective functional in (\ref{cost}) is defined in terms of the Tikhonov regularization parameter $\alpha>0$ and a given {deterministic, again defined in a hold-all domain,} target state ${y_d} \in L^2(\mathcal{D}(\omega))$, is to be achieved through the action of the random field $u(\omega)$ representing a distributed control.

\subsection{Weak form}
Throughout the article, we adopt the notation  {$(v,w)_{\mathcal{X}}=\int_\mathcal{X} vw \,d\mathcal{X}$} 
for the usual $L^2( {\mathcal{X}})$-inner product on $ {\mathcal{X}}\subset\mathbb{R}^n$ ($n=1, 2, 3$), 
with induced norm  {$\|v\|_{ {\mathcal{X}}}=( v,v)_{\mathcal{X}}^{1/2}=\left(\int_{ {\mathcal{X}}}|v|^2 d\mathcal{X}\right)^{1/2}$} and  {$\|v\|_{\partial \mathcal{X}}=\langle v,v \rangle_{\partial \mathcal{X}}^{1/2}=\left(\int_{\partial \mathcal{X}}|v|^2 d\mathcal{X}\right)^{1/2}$}. Let also  {$\|\cdot\|_{s,\mathcal{X}}$} and  {$|\cdot|_{s,\mathcal{X}}$} denote the standard  {$H^s(\mathcal{X})$} ($s\in\mathbb{R}$) Sobolev space norm and semi-norm, respectively. Notice that  {$\|\cdot\|_{0,\mathcal{X}}=\|\cdot\|_{\mathcal{X}}$}.
Defining the solution and test spaces  as  
$V_{g_D}(\omega) = \Big\{ w\in H^1(\mathcal{D}(\omega)):\left. w \right|_{\Gamma_D(\omega)} = {g_D}\Big\}$ and 
$V_0(\omega) = \Big\{ w\in H^1(\mathcal{D}(\omega)):\left. w \right|_{\Gamma_D(\omega)} = 0 \Big\}$ respectively, the weak formulation of  (\ref{mixedBVP}) reads: find $y\in V_{g_D}(\omega)$ such that for almost every $\omega\in\Omega$ 
\begin{equation}\label{variational}
a(y(\omega),v(\omega)) = ( { f}+u(\omega), v(\omega) )_{\mathcal{D}(\omega)} + \langle { g_N}, v(\omega)\rangle_{\Gamma_N(\omega)}, 
\,\,\,\forall \,\,\, v(\omega) \in V_0(\omega),
\end{equation}
with bilinear form $a(y(\omega),v(\omega)):= ( \nabla y(\omega),\nabla v(\omega) )_{\mathcal{D}(\omega)}$. 
In order to study the sensitivity of the control problem (\ref{cost})-(\ref{variational}) and obtain a robust control with respect to random fluctuations in the parameterized domain $\mathcal{D}(\omega)$, we consider the function space
\[
V_{P,g}=\left\{ y: \omega\in \Omega \to y(\cdot,\omega)\in V_{g_{D}}(\omega); \int_{\Omega}\|y\|_{V_{g_{D}}(\omega)}^2 dP(\omega)<\infty \right\}
\]
and we seek $y(\omega)\in V_{P,g}$ such that $\forall\,\, v(\omega)\in V_{P,0}$
\begin{equation}\label{SOCP}
\int_{\mathcal{D}}\mathbb{E}\left[\nabla y(\omega)\cdot\nabla v(\omega)\right]dx  = 
\int_{\mathcal{D}}\mathbb{E}\left[({ f} + u(\omega))v(\omega) \right]dx  + 
\int_{\Gamma_N}\mathbb{E}\left[ { g_N} v(\omega) \right]dx.
\end{equation}
\subsection{Optimality system}
A random field can be expressed as a function
of random variables.  This fact enables us to formulate the optimal control problem under uncertain geometries for  the  cost  functional (\ref{cost}) and  the  constraints in (\ref{mixedBVP}) as a parametric PDE-constrained optimization problem. Hence
for a fixed parameter $\omega\in\Omega$, the optimal control problem inherits directly the properties from its deterministic infinite dimensional counterpart. Due to the convexity of the functional and the linearity of the state equation, the problem (\ref{cost})-(\ref{variational}) is amenable to a well established existence and uniqueness theory \cite{AUH2017,GKKSS2019, L1971, T2010}. Its solution is uniquely characterized by the following first-order necessary and sufficient optimality conditions for a fixed outcome $\omega\in\Omega$.

\begin{theorem}\label{continuous}
The problem (\ref{cost}),(\ref{variational}) admits a unique optimal control $\bar{u}(\omega)\in L^2(\mathcal{D}(\omega))$,
with an associated state $\bar{y}(\omega)\in H^1(\mathcal{D}(\omega))$ and an adjoint state 
$\bar{p}(\omega)\in H^1(\mathcal{D}(\omega))$ that satisfy the optimality conditions
\begin{eqnarray}\label{OCP}
a(\bar{y}(\omega), v(\omega)) & = & ( { f}+\bar{u}(\omega),v(\omega) )_{\mathcal{D}(\omega)} +  
\langle { g_N},v(\omega)\rangle_{\Gamma_N(\omega)}, 
\label{OCPa} \\ 
a(v(\omega), \bar{p}(\omega)) & = & ( \bar{y}(\omega) - { y_d}, v(\omega) )_{\mathcal{D}(\omega)}, 
\label{OCPb} \\ 
\alpha \bar{u}(\omega) + \bar{p}(\omega)  & =& 0, \label{OCPc}
\end{eqnarray}
for any $v(\omega)\in V_0(\omega)$.
Moreover, the adjoint equation (\ref{OCPb}) is the weak formulation of the 
homogeneous mixed boundary value problem
\begin{eqnarray}\label{adjointBVP}
\nonumber - \Delta \bar{p}(\omega) & = & \bar{y}(\omega) - { y_d} \quad \,\,\,\textrm{in}\quad \mathcal{D}(\omega), \\
\bar{p}(\omega) & = & 0 \quad \quad \,\,\,\textrm{on}\quad \Gamma_D(\omega),\\
\nonumber \mathbf{n}_\Gamma\cdot\nabla \bar{p}(\omega) & = & 0 \quad \quad \,\,\,\textrm{on}\quad \Gamma_N(\omega).
\end{eqnarray}
\end{theorem}
Additional regularity for the optimal triple $(\bar{y}(\omega),\bar{p}(\omega),\bar{u}(\omega))$ for some fixed parameter $\omega\in\Omega$ is required for its numerical approximation  via the cut elements FEM. { We assume sufficient regular behavior of the domain and data \cite{Gri85} such that the elliptic regularity shift theorem holds}:

\begin{theorem}\label{regularity}
{Let $\mathcal{D}(\omega)\subset\mathbb{R}^2$ be an open bounded domain for any outcome $\omega\in\Omega$ with boundary of class $C^{1,1}$}. If the data ${ f} \in L^2(\mathcal{D}(\omega))$, ${ g_D}\in H^{\frac{3}{2}}(\Gamma_D(\omega))$ and 
${ g_N}\in H^{\frac{1}{2}}(\Gamma_N(\omega))$, then the solution $(\bar{y}(\omega),\bar{p}(\omega),\bar{u}(\omega))$ of the problem (\ref{OCPa})-(\ref{OCPc}) satisfies the additional regularity $(\bar{y}(\omega),\bar{p}(\omega),\bar{u}(\omega))\in H^2(\mathcal{D}(\omega))\times H^2(\mathcal{D}(\omega))\times L^2(\mathcal{D}(\omega))$. In particular, the estimates
\begin{eqnarray}
\|\bar{y}(\omega)\|_{2,\mathcal{D}(\omega)} & \leq & C_1 \left(  \|\bar{u}(\omega)\|_{0,\mathcal{D}(\omega)} + C_{\omega}({ f}, {g_D},{ g_N})\right),\\
\|\bar{p}(\omega)\|_{2,\mathcal{D}(\omega)} & \leq & C_2 \left( \|\bar{y}(\omega)\|_{0,\mathcal{D}(\omega)} + \|{ y_d}\|_{0,\mathcal{D}(\omega)}\right),
\end{eqnarray}
with $C_{\omega}({ f, g_D, g_N})=\|{ f}\|_{0,\mathcal{D}(\omega)} + \|{ g_D}\|_{\frac{3}{2},\Gamma_D(\omega)} + \|{ g_N}\|_{\frac{1}{2},\Gamma_N(\omega)}$, are satisfied for some positive constants $C_1=C_1(\mathcal{D}(\omega))$, $C_2=C_2(\mathcal{D}(\omega))$ which depend on $\mathcal{D}(\omega)$.
\end{theorem}

For the numerical approximation to the solution of the optimal control problem (\ref{cost}),(\ref{SOCP}), we will employ both an 
unfitted finite element discretization method with cut elements for deterministic approximation in the spatial domain $\mathcal{D}(\omega)$ and a Quasi Monte Carlo approximation in the probability domain $\Omega$, as it is described in the following sections.

\section{Unfitted finite elements discrete form}
\label{section3}

The Cut Finite Element Method  \cite{BH2012} discretization of the continuous state and adjoint equations in  (\ref{OCPa})-(\ref{OCPb}) requires the definition of a suitable shape--regular background mesh $\widetilde{\mathcal{T}}_h\subset\mathbb{R}^2$ which contains the original domain of interest $\mathcal{D}(\omega)$, but it is not fitted to its boundary. In this direction,  $\mathcal{D}(\omega)$ is embedded into a larger fictitious domain $\widetilde{\mathcal{D}}\subset\mathbb{R}^2$ of a simpler, usually polygonal geometry. Its boundary $\Gamma(\omega)=\partial\mathcal{D}(\omega)$ is implicitly described by a zero level set function $\phi:\widetilde{\mathcal{D}}\times\Omega\to\mathbb{R}$; i.e., $\Gamma(\omega)=\{x\in\widetilde{\mathcal{D}}: \phi(x,\omega)=0\}$ and 
$\mathcal{D}(\omega)=\{x\in\widetilde{\mathcal{D}}: \phi(x, \omega)<0 \}$. Then, $\widetilde{\mathcal{T}}_h\subseteq\widetilde{\mathcal{D}}$ may be defined as a conforming quasi-uniform triangulation of $\widetilde{\mathcal{D}}$ consisting of shape--regular 
simplices $K$.
The subscript  $h=\max_{K\in\widetilde{\mathcal{T}}_h}diam(K)$ indicates the mesh size and $\mathcal{D}_{\mathcal{T}}$ is the \emph{extended} domain which is covered by the \textit{active} part $\mathcal{T}_h=\{K\in\widetilde{\mathcal{T}}_h: K\cap \mathcal{D}(\omega)\neq\emptyset\}$ of the background mesh $\widetilde{\mathcal{T}}_h$; this situation is illustrated in Figure \ref{cutfig}. In the following, discrete solutions will actually be calculated in a { sample-dependent subspace of the finite element space on the background mesh:}
\begin{equation}\label{space}
{ V^h(\omega)} =\left \{  \upsilon \in C^0(\bar {\mathcal{D}}_{\mathcal T})\,:\,\upsilon |_K  \in P^1(K), \, \forall K\in \mathcal{T}_h \right \}.
\end{equation}

The boundary conditions at $\Gamma(\omega)$ will be satisfied weakly through a variant of Nitsche's method. On the other hand, coercivity over the whole computational domain $\mathcal{D}_{\mathcal{T}}$  is ensured by means of additional ghost penalty terms  which act on the gradient jumps in the boundary zone. Therefore, a more delicate analysis of the boundary grid is required; using the  notation ${G}_h:=\{K\in \widetilde{\mathcal{T}}_h:K\cap\Gamma(\omega)\neq\emptyset\}$ for the set of elements which are intersected by the interface, the following assumptions are required:

\begin{itemize}
\item[(A)] $\Gamma(\omega)$ intersects the boundary $\partial K$ of each cut 
element $K\in G_h$ exactly twice and each edge of $K\in G_h$ at most once. 

\item[(B)] Given a piecewise linear approximation $\Gamma_h(\omega)$ of $\Gamma(\omega)$,
there is a piecewise smooth function which maps $\Gamma_h(\omega)\cap K$ to $\Gamma(\omega)\cap K$
for each cut element $K\in G_h$. 

\item[(C)] For any cut element $K\in G_h$ there exists a neighboring element 
$K'\in { \mathcal{T}_h}$ such that $K'\notin G_h$ and $K\cap K'\neq\emptyset$. The measures 
of the elements $K$, $K'$ and the faces $F$ such that $K\cap F\neq\emptyset$ and 
$K'\cap F\neq\emptyset$ are comparable.
\end{itemize}

The set of element faces associated with $G_h$ is defined by 
\[
\mathcal{F}_G = \left\{F = K \cap K':\,\,\,\textrm{at least one of } K, K' \in G_h \right\}\neq\emptyset.
\] 
The latter definition ensures that all boundary faces of 
$\mathcal{D}_{\mathcal{T}}$ are excluded from $\mathcal{F}_G$.

The  jump of the gradient of $v\in { V^h(\omega)}$ on some  face $F=K\cap K' \in \mathcal{F}_G$ between neighboring elements
is denoted by
$\left[ \kern-0.15em \left[\mathbf{n}_F \cdot \nabla v \right] \kern-0.15em \right] = 
\left. \mathbf{n}_F\cdot\nabla v \right|_K - \left. \mathbf{n}_F \cdot \nabla v \right|_{K'}$, where $\mathbf{n}_F$ is a fixed, but arbitrary,
unit normal to the facet $F$.
\begin{figure}[h]
\centering
\includegraphics[width=0.7\textwidth]{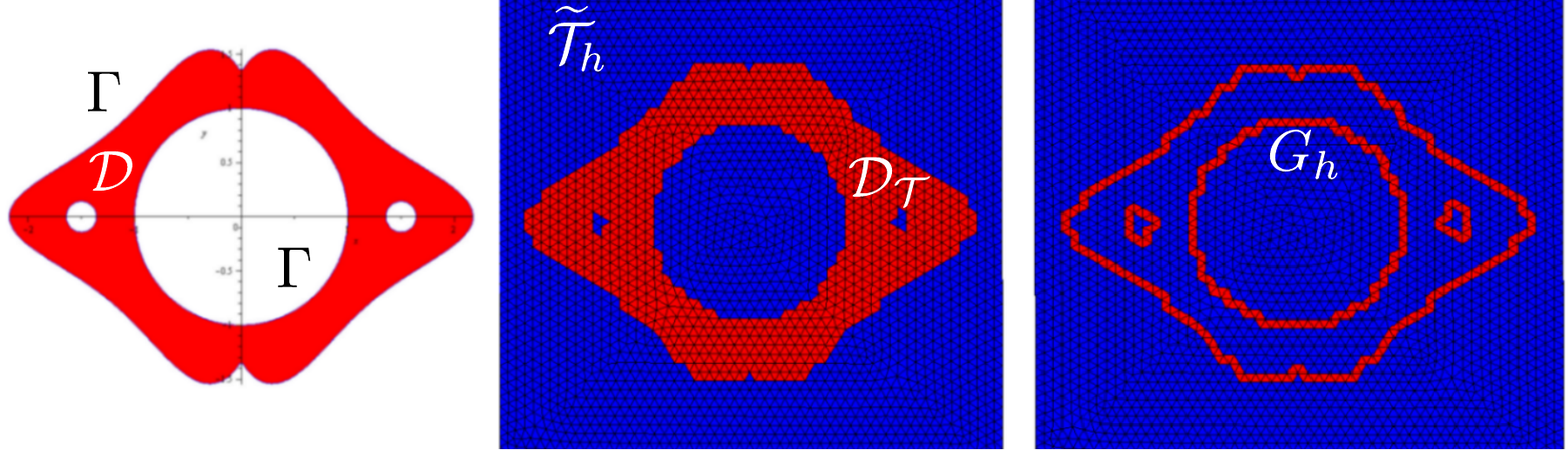}
\caption{The original spatial domain $\mathcal{D}$ (left picture) and its boundary $\Gamma$ are represented implicitly by the level-set function $\phi(x,\omega)$ in (\ref{complex_geometry}), and for $\omega=(\omega_1,\omega_2)=(9,2)$ they are designated by the red colored area. The extended computational domain $\mathcal{D}_{\mathcal{T}}$ is visualized in the middle picture, and it is covered by the active part of the background mesh $\widetilde{\mathcal{T}}_h$ colored in red. The subset $G_{h}$ of elements in $\widetilde{\mathcal{T}}_h$ that intersect the boundary $\Gamma$ is shown in red at the right picture.}\label{cutfig}
\end{figure}
\begin{figure}[h]
\centering
\includegraphics[width=0.75\textwidth]{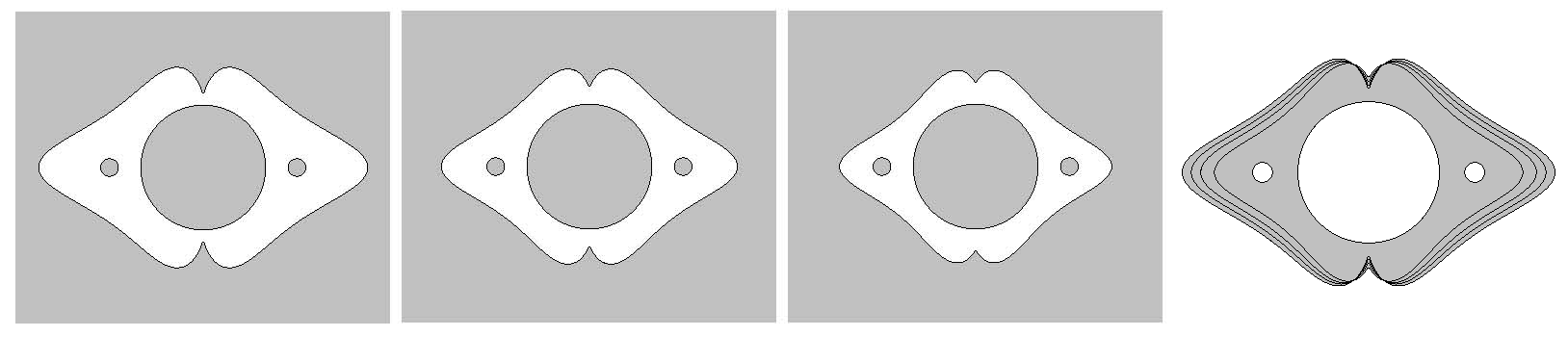}
\caption{Visualization of some parameterized geometry deformations (uncovered areas) prescribed by the level-set function $\phi(x,\omega)$ in (\ref{complex_geometry}) with respect to three samples $\omega=(\omega_1,\omega_2)\in\{(10,2.4), (11, 2.7), (9,2)\}$ from left to right. The rightmost image shows the range of the geometric variation (grey areas).}\label{parametrization}
\end{figure}

The  cut elements FEM discretizations for the state and adjoint equations (\ref{OCPa})--(\ref{OCPb}) now read as follows: find $y^h(\omega), p^h(\omega)\in { V^h(\omega)}$ such that
\begin{subequations}
\begin{eqnarray}\label{DOCP}
A_h(y^h(\omega), { v_h}) & = & ({\bar{u}(\omega), { v_h}})_{\mathcal{D}(\omega)} + L_h({ v_h}),  \quad \forall\,\, {\ v_h}\in { V^h(\omega)} \label{DOCPa}\\
A_h({ v_h}, p^h(\omega)) & = & ({\bar{y}(\omega) - { y_d}, { v_h}})_{\mathcal{D}(\omega)},  \quad\quad \forall\,\, { v_h}\in { V^h(\omega)}. \label{DOCPb}
\end{eqnarray}
\end{subequations}
Here, the discrete bilinear and linear forms $A_h(\cdot, \cdot)$, $L_h(\cdot)$, respectively, are defined by
\begin{eqnarray*}\label{DBF}
A_h(y^h(\omega), {\ v_h}) & = & a_h(y^h(\omega),{ v_h}) + j(y^h(\omega), {\ v_h}) \\
A_h({ v_h}, p^h(\omega)) & = & a_h(p^h(\omega),{ v_h}) + j(p^h(\omega), { v_h}) \\
L_h({ v_h}) & = & ({\ f}, { v_h})_{\mathcal{D}(\omega)} + 
\langle{{g_D}, \gamma_D(\omega) h^{-1}{ v_h} - \mathbf{n}_\Gamma\cdot\nabla { v_h}}\rangle_{\Gamma_D(\omega)} \\
& & \,\,\, + \langle{{\ g_N}, { v_h} + \gamma_N h \mathbf{n}_\Gamma\cdot\nabla { v_h}}\rangle_{\Gamma_N(\omega)} 
\end{eqnarray*}
in terms of the modified discrete bilinear form
\begin{eqnarray*}
a_h(w^h(\omega),{ v_h}) & = & ({\nabla w^h(\omega), \nabla {v_h}})_{\mathcal{D}(\omega)} -
\langle{\mathbf{n}_\Gamma\cdot\nabla w^h(\omega), { v_h}}\rangle_{\Gamma_D(\omega)} \\
&   & \qquad - \langle{\mathbf{n}_\Gamma\cdot\nabla {\ v_h}, w^h(\omega)}\rangle_{\Gamma_D(\omega)} + \langle{\gamma_D h^{-1} w^h(\omega), { v_h}}\rangle_{\Gamma_D(\omega)} \\
&   & \qquad + \langle{\gamma_N h\mathbf{n}_\Gamma\cdot \nabla w^h(\omega), \mathbf{n}_\Gamma\cdot \nabla {v_h}}\rangle_{\Gamma_N(\omega)}
\end{eqnarray*}
and the stabilization term
\[
j(w^h(\omega), {\ v_h})  =  \mathop\sum\limits_{F\in \mathcal{F}_G}{\langle{\gamma_1 h[\kern-0.15em[\mathbf{n}_F\cdot\nabla w^h(\omega) ]\kern-0.15em], \left[\kern-0.15em\left[\mathbf{n}_F\cdot\nabla {\ v_h}\right]\kern-0.15em\right]}\rangle_F},
\]
that extends coercivity from the physical domain $\mathcal{D}(\omega)$ to $\mathcal{D}_{\mathcal{T}}$ ($w^h(\omega)=y^h(\omega), p^h(\omega)$). The quantities $\gamma_D$, $\gamma_N$ and $\gamma_1$ are positive penalty parameters. 

We remark that the auxiliary problems (\ref{DOCPa})-(\ref{DOCPb}) are \emph{uncoupled} equations. 
We recall \cite{BH2012} where the discrete space ${ V^h(\omega)}$ approximates  functions in $H^2(\mathcal{D}(\omega))$ optimally with  respect to the following mesh dependent norms:
\begin{eqnarray}
\left\vert\kern-0.25ex\left\vert\kern-0.25ex\left\vert {v(\omega)}
\right\vert\kern-0.25ex\right\vert\kern-0.25ex\right\vert_{*}^2  & := & \|\nabla v(\omega)\|_{0,\mathcal{D}(\omega)}^2 + 
\|h^{\frac{1}{2}}\mathbf{n}_\Gamma\cdot\nabla v(\omega)\|_{0,\Gamma_N(\omega)}^2 + 
\|\gamma_D^{\frac{1}{2}}h^{-\frac{1}{2}} v(\omega)\|_{0,\Gamma_D(\omega)}^2, \label{star_norm}\\
\left\vert\kern-0.25ex\left\vert\kern-0.25ex\left\vert {v(\omega)}
\right\vert\kern-0.25ex\right\vert\kern-0.25ex\right\vert_{h}^2 & := & \|\nabla v(\omega)\|_{0,\mathcal{D}_{\mathcal{T}}}^2 + 
\|\gamma_N^{\frac{1}{2}}h^{\frac{1}{2}}\mathbf{n}_\Gamma\cdot\nabla v(\omega)\|_{0,\Gamma_N(\omega)}^2 + 
\|\gamma_D^{\frac{1}{2}}h^{-\frac{1}{2}} v(\omega)\|_{0,\Gamma_D(\omega)}^2 \nonumber\\
& & \qquad\qquad\qquad\qquad\qquad\qquad\qquad\qquad\qquad\,\,\,
+ j(v(\omega),v(\omega)).\nonumber
\end{eqnarray}
Notice that $\left\vert\kern-0.25ex\left\vert\kern-0.25ex\left\vert {v(\omega)} 
\right\vert\kern-0.25ex\right\vert\kern-0.25ex\right\vert_{*}\lesssim \left\vert\kern-0.25ex\left\vert\kern-0.25ex\left\vert {v(\omega)}
\right\vert\kern-0.25ex\right\vert\kern-0.25ex\right\vert_{h}$ for $v(\omega)\in H^2({\mathcal{T}_h})$, where the relation 
$a\lesssim b$ signifies that there exists a generic constant $C>0$ that is 
always independent of $h$ such that $a\leq C b$.

The following result, \cite[Lemmas 6 and 7]{BH2012}, verifies that 
the bilinear form $A_h(\cdot, \cdot)$ is coercive and bounded with 
respect to the above norms. 

\begin{lemma}\label{lem3.1}
Let $\gamma_D$ be sufficiently large and $\gamma_1=1$, $\gamma_N\geq0$. Then 
for any $w_h(\omega), v_h(\omega) \in { V^h(\omega)}$
\[
A_h(w_h(\omega), w_h(\omega)) \gtrsim  \left\vert\kern-0.25ex\left\vert\kern-0.25ex\left\vert {w_h(\omega)}
\right\vert\kern-0.25ex\right\vert\kern-0.25ex\right\vert_h^2  \quad\textrm{and}\quad
A_h(w_h(\omega), v_h(\omega)) \lesssim  \left\vert\kern-0.25ex\left\vert\kern-0.25ex\left\vert{w_h(\omega)}
\right\vert\kern-0.25ex\right\vert\kern-0.25ex\right\vert_h \left\vert\kern-0.25ex\left\vert\kern-0.25ex\left\vert{v_h(\omega)}
\right\vert\kern-0.25ex\right\vert\kern-0.25ex\right\vert_h,
\]
independently of $h$ and of how the triangulation is intersected by $\Gamma(\omega)$.
\end{lemma}

\begin{lemma}\label{lemma3.2}
Let $(\bar{y}(\omega),\bar{p}(\omega),\bar{u}(\omega))\in H^2(\mathcal{D}(\omega))\times H^2(\mathcal{D}(\omega))\times L^2(\mathcal{D}(\omega))$
be the optimal solution of the continuous problem \eqref{OCPa}-\eqref{OCPb} for a fixed parameter $\omega\in\Omega$. Then, for penalty parameters 
$\gamma_1=1$, $\gamma_N\geq0$ and $\gamma_D$ sufficiently large,  the uncoupled equations \eqref{DOCPa}-\eqref{DOCPb} admit 
unique solutions $y^h(\omega), p^h(\omega)\in { V^h(\omega)}$, such that the following a--priori error estimates hold:
\begin{eqnarray}
\vert\kern-0.25ex\vert\kern-0.25ex\vert \bar{y}(\omega)-y^h(\omega) \vert\kern-0.25ex\vert\kern-0.25ex\vert_{*} & \lesssim & h\left(\|\bar{u}(\omega)\|_{0,\mathcal{D}(\omega)} + C_{\omega}({ f,g_D,g_N})\right) \label{lem4.2:a}
\\
\vert\kern-0.25ex\vert\kern-0.25ex\vert \bar{p}(\omega)-p^h(\omega) \vert\kern-0.25ex\vert\kern-0.25ex\vert_{*} & \lesssim & h\left(\|\bar{u}(\omega)\|_{0,\mathcal{D}(\omega)} + \|{ y_d}\|_{0,\mathcal{D}(\omega)}\right. \label{lem4.2:b}  + C_{\omega}({ f,g_D,g_N})) \\
\|\bar{y}(\omega)-y^h(\omega)\|_{0,\mathcal{D}(\omega)} & \lesssim & h^2 \left(\|\bar{u}(\omega)\|_{0,\mathcal{D}(\omega)} + C_{\omega}({ f,g_D,g_N})\right) \label{lem4.2:c}\\
\|\bar{p}(\omega)-p^h(\omega)\|_{0,\mathcal{D}(\omega)} & \lesssim & h^2 \left(\|\bar{u}(\omega)\|_{0,\mathcal{D}(\omega)} + \|{ y_d}\|_{0,\mathcal{D}(\omega)}  
+ C_{\omega}({\ f,g_D,g_N})\right),\label{lem4.2:d}
\end{eqnarray}
\end{lemma}
\begin{proof}
The estimate \eqref{lem4.2:a}  follows readily by \cite[Corollary 9]{BH2012} and Theorem \ref{regularity},  while \eqref{lem4.2:b} 
is due to
\begin{eqnarray*}
\vert\kern-0.25ex\vert\kern-0.25ex\vert \bar{p}(\omega)-p^h(\omega) \vert\kern-0.25ex\vert\kern-0.25ex\vert_{*} & \lesssim & h\left(\|\bar{y}(\omega)\|_{2,\mathcal{D}(\omega)} + 
\|{y_d}\|_{0,\mathcal{D}(\omega)} \right) \\
& \lesssim & h\left(\|\bar{u}(\omega)\|_{0,\mathcal{D}(\omega)} + \|{y_d}\|_{0,\mathcal{D}(\omega)} + C_{\omega}({ f,g_D,g_N})\right),
\end{eqnarray*}
where the second inequality follows by the Sobolev embedding theorem.
The proof of \eqref{lem4.2:c}--\eqref{lem4.2:d} is analogous, invoking \cite[Proposition 10]{BH2012}.
\end{proof}

\section{Discrete optimality system and deterministic error analysis}
\label{section4}
A discrete analogue of the parameterized optimization problem \eqref{cost}-\eqref{SOCP} may be obtained through the variational discretization concept introduced by Hinze in \cite{H2005} for any realization $\omega\in\Omega$. The basic idea of this approach is to discretize only the state which results in a new optimal control problem defined by  
\begin{equation}\label{dcost}
\min \mathcal{J}_h(y_h(\omega),u(\omega)) := \frac{1}{2}\|y_h(\omega) - {y_d}\|_{0,\mathcal{D}(\omega)}^2 + 
\frac{\alpha}{2}\| u(\omega)\|_{0,\mathcal{D}(\omega)}^2
\end{equation}
over all $(y_h(\omega),u(\omega))\in { V^h(\omega)}\times L^2(\mathcal{D}(\omega))$, where $y_h(\omega)=y_h(\cdot,\omega)$ satisfies
\begin{equation}\label{dvariational}
A_h(y_h(\omega), { v_h}) =({u(\omega), { v_h}})_{\mathcal{D}(\omega)} + L_h({ v_h}),  \quad \forall\,\, {\ v_h}\in { V^h(\omega)}.
\end{equation}
Arguing as in Theorem \ref{continuous}, analogous existence 
and uniqueness statements hold for a fixed parameter $\omega\in\Omega$.

\begin{lemma}\label{discrete}
The problem \eqref{dcost}-\eqref{dvariational} admits a unique solution  $(\bar{y}_h(\omega), \bar{u}_h(\omega))\in { V^h(\omega)}\times L^2(\mathcal{D}(\omega))$  with  associated   adjoint state $\bar{p}_h(\omega)\in { V^h(\omega)}$ that satisfy the optimality conditions
\begin{subequations}
\begin{eqnarray}\label{dOCP}
A_h(\bar{y}_h(\omega), { v_h}) & = & ({\bar{u}_h(\omega), {v_h}})_{\mathcal{D}(\omega)} + L_h({ v_h}),  \quad \forall\,\, {v_h}\in {V^h(\omega)} \label{dOCPa} \\ 
A_h({ v_h}, \bar{p}_h(\omega)) & = &  ( \bar{y}_h(\omega) - {y_d}, { v_h} )_{\mathcal{D}(\omega)}, \quad \forall\,\, { v_h}\in { V^h(\omega)} \label{dOCPb} \\ 
\alpha\bar{u}_h(\omega) + \bar{p}_h(\omega) & = & 0.
\quad\quad\quad\quad\quad\,\,   \label{dOCPc}
\end{eqnarray}
\end{subequations}
\end{lemma}

Hence, although  the control space  is not directly discretized in \eqref{dcost}, { the} optimality condition \eqref{dOCPc} enforces an implicit discretization on the control   through the adjoint function. 

For a fixed realization $\omega\in\Omega$ the next result provides error estimates between the solutions $(\bar{y}(\omega), \bar{p}(\omega), \bar{u}(\omega))$ and 
$(\bar{y}_h(\omega), \bar{p}_h(\omega), \bar{u}_h(\omega))$ of the continuous 
optimal control problem \eqref{OCPa}-\eqref{OCPc} and its discrete counterpart 
\eqref{dOCPa}-\eqref{dOCPc}, respectively.

\begin{theorem}\label{errors1}
Let $(\bar{y}(\omega), \bar{p}(\omega), \bar{u}(\omega))\in H^2(\mathcal{D}(\omega))\times H^2(\mathcal{D}(\omega))\times L^2(\mathcal{D}(\omega))$
and $(\bar{y}_h(\omega), \bar{p}_h(\omega), \bar{u}_h(\omega))\in {V^h(\omega)} \times { V^h(\omega)}\times L^2(\mathcal{D}(\omega))$ be the
solutions of the continuous optimal control problem \eqref{OCPa}-\eqref{OCPc} 
and its discrete counterpart \eqref{dOCPa}-\eqref{dOCPc}, respectively, for a fixed parameter $\omega\in\Omega$. Then 
we obtain
\begin{align}\label{eq4.4}
\alpha\|\bar{u}(\omega)-\bar{u}_h(\omega)\|_{0,\mathcal{D}(\omega)}^2 + \|\bar{y}(\omega)-\bar{y}_h(\omega)\|_{0,\mathcal{D}(\omega)}^2 
\leq \frac{1}{\alpha}\|\bar{p}(\omega)-p^h(\omega)\|_{0,\mathcal{D}(\omega)}^2 
\\ \nonumber \quad
+ \|\bar{y}(\omega)-y^h(\omega)\|_{0,\mathcal{D}(\omega)}^2 \\
\|\bar{p}(\omega)-\bar{p}_h(\omega)\|_{0,\mathcal{D}(\omega)} \lesssim 
\|\bar{p}(\omega)-p^h(\omega)\|_{0,\mathcal{D}(\omega)} + \|\bar{y}(\omega)-\bar{y}_h(\omega)\|_{0,\mathcal{D}(\omega)} \label{eq4.5} \\
\left\vert\kern-0.25ex\left\vert\kern-0.25ex\left\vert{\bar{y}(\omega)-\bar{y}_h(\omega)}
\right\vert\kern-0.25ex\right\vert\kern-0.25ex\right\vert_{*} \lesssim 
\vert\kern-0.25ex\vert\kern-0.25ex\vert \bar{y}(\omega)-y^h(\omega) \vert\kern-0.25ex\vert\kern-0.25ex\vert_{*} + \|\bar{u}(\omega)-\bar{u}_h(\omega)\|_{0,\mathcal{D}(\omega)} \label{eq4.6} \\
\left\vert\kern-0.25ex\left\vert\kern-0.25ex\left\vert{\bar{p}(\omega)-\bar{p}_h(\omega)} 
\right\vert\kern-0.25ex\right\vert\kern-0.25ex\right\vert_{*} \lesssim 
\vert\kern-0.25ex\vert\kern-0.25ex\vert \bar{p}(\omega)-p^h(\omega) \vert\kern-0.25ex\vert\kern-0.25ex\vert_{*} + \|\bar{y}(\omega)-\bar{y}_h(\omega)\|_{0,\mathcal{D}(\omega)}, \label{eq4.7} 
\end{align}
where $y^h(\omega), p^h(\omega)\in { V^h(\omega)}$ are the cut elements FEM solutions of the discrete
expressions \eqref{DOCPa}-\eqref{DOCPb}.
\end{theorem}
\begin{proof}
We note that the optimality conditions \eqref{OCPc}, \eqref{dOCPc} imply
\begin{equation*}
({\alpha(\bar{u}(\omega)-\bar{u}_h(\omega)) + \bar{p}(\omega)-\bar{p}_h(\omega), \bar{u}_h(\omega)-\bar{u}(\omega)})_{\mathcal{D}(\omega)} =0.
\end{equation*}
Hence, the expression
\begin{eqnarray}
\nonumber \alpha \|\bar{u}(\omega)-\bar{u}_h(\omega)\|_{0,\mathcal{D}(\omega)}^2 & = & ({\bar{p}(\omega)-\bar{p}_h(\omega), \bar{u}_h(\omega)-\bar{u}(\omega)})_{\mathcal{D}(\omega)}\\
\nonumber & = & ({\bar{p}(\omega)-p^h(\omega), \bar{u}_h(\omega)-\bar{u}(\omega)})_{\mathcal{D}(\omega)}\\ 
&   & + ({p^h(\omega)-\bar{p}_h(\omega), \bar{u}_h(\omega)-\bar{u}(\omega)})_{\mathcal{D}(\omega)},  \label{new}
\end{eqnarray}
follows, where the interjected term $p^h(\omega)\in { V^h(\omega)}$ is the CutFEM approximate solution of the auxiliary problem \eqref{DOCPb}.

We next estimate the last two terms separately. For the first one, the Cauchy-Schwarz and Young inequalities yield
\begin{equation}\label{eq4.11}
({\bar{p}(\omega)-p^h(\omega), \bar{u}_h(\omega)-\bar{u}(\omega)})_{\mathcal{D}(\omega)} \leq
\frac{1}{2\alpha}\|\bar{p}(\omega)-p^h(\omega)\|_{0,\mathcal{D}(\omega)}^2 + \frac{\alpha}{2}\|\bar{u}_h(\omega)-\bar{u}(\omega)\|_{0,\mathcal{D}(\omega)}^2. 
\end{equation}
For the second term, combining \eqref{DOCPa} 
with \eqref{dOCPa} and \eqref{DOCPb} with \eqref{dOCPb}, we obtain
\begin{eqnarray}
A_h(\bar{y}_h(\omega)-y^h(\omega), { v_h}) & = & ({\bar{u}_h(\omega)-\bar{u}(\omega), { v_h}})_{\mathcal{D}(\omega)} \quad \forall\,\, { v_h}\in { V^h(\omega)} \label{eq4.8}\\
A_h({v_h}, \bar{p}_h(\omega)-p^h(\omega)) & = & ({\bar{y}_h(\omega)-\bar{y}(\omega), { v_h}})_{\mathcal{D}(\omega)} \quad \forall\,\, {\ v_h}\in { V^h(\omega)}. \label{eq4.9}
\end{eqnarray} 
Testing \eqref{eq4.8} by $p^h(\omega)-\bar{p}_h(\omega)\in {V^h(\omega)}$, \eqref{eq4.9} by $\bar{y}_h(\omega)-y^h(\omega)\in {V^h(\omega)}$, 
adding and subtracting $\bar{y}(\omega)$ and using the polarization identity, we can conclude the bound 
\begin{eqnarray}
\nonumber ({p^h(\omega)-\bar{p}_h(\omega), \bar{u}_h(\omega)-\bar{u}(\omega)})_{\mathcal{D}(\omega)} & = & A_h(\bar{y}_h(\omega)-y^h(\omega), p^h(\omega)-\bar{p}_h(\omega)) \\
\nonumber & = & ({\bar{y}(\omega) - \bar{y}_h(\omega), \bar{y}_h(\omega)-y^h(\omega)})_{\mathcal{D}(\omega)} \\
& \leq & - \frac{1}{2}\|\bar{y}(\omega)- \bar{y}_h(\omega)\|_{0,\mathcal{D}(\omega)}^2 
+ \frac{1}{2}\|\bar{y}(\omega) -y^h(\omega)\|_{0,\mathcal{D}(\omega)}^2.  \label{eq4.12}
\end{eqnarray}
Thus, estimate \eqref{eq4.4} readily follows from \eqref{new}, \eqref{eq4.11} and \eqref{eq4.12}.

For \eqref{eq4.5}, we first consider the triangle inequality
\begin{equation}\label{eq4.13}
\|\bar{p}(\omega) -\bar{p}_h(\omega)\|_{0,\mathcal{D}(\omega)} \leq \|\bar{p}(\omega) -p^h(\omega)\|_{0,\mathcal{D}(\omega)}+\|\bar{p}_h(\omega) -p^h(\omega)\|_{0,\mathcal{D}(\omega)}.
\end{equation}
Since $\bar{p}_h(\omega) -p^h(\omega) \in { V^h(\omega)}$, using the coercivity property in 
Lemma \ref{lem3.1} and \eqref{eq4.9}, we have
\begin{eqnarray}
\nonumber \|\bar{p}_h(\omega) - p^h(\omega)\|_{0,\mathcal{D}(\omega)}^2 & \lesssim &
\vert\kern-0.25ex\vert\kern-0.25ex\vert \bar{p}_h(\omega)-p^h(\omega) \vert\kern-0.25ex\vert\kern-0.25ex\vert_{h}^2 \\
\nonumber  & \lesssim & A_h(\bar{p}_h(\omega) - p^h(\omega),\bar{p}_h(\omega) - p^h(\omega)) \\
\nonumber  & = & ({\bar{y}_h(\omega)-\bar{y}(\omega), \bar{p}_h(\omega) - p^h(\omega)})_{\mathcal{D}(\omega)} \\   
& \lesssim & \|\bar{y}_h(\omega)-\bar{y}(\omega)\|_{0,\mathcal{D}(\omega)}\, \|\bar{p}_h(\omega) - p^h(\omega)\|_{0,\mathcal{D}(\omega)} \label{eq4.14} 
\end{eqnarray}
Combining \eqref{eq4.13} with \eqref{eq4.14}, the estimate \eqref{eq4.5} holds.

For the estimate \eqref{eq4.6} we can consider the triangle inequality
\begin{equation}\label{eq4.15}
\left\vert\kern-0.25ex\left\vert\kern-0.25ex\left\vert{\bar{y}(\omega)-\bar{y}_h(\omega)}
\right\vert\kern-0.25ex\right\vert\kern-0.25ex\right\vert_{*}  \lesssim 
\vert\kern-0.25ex\vert\kern-0.25ex\vert \bar{y}(\omega)-y^h(\omega) \vert\kern-0.25ex\vert\kern-0.25ex\vert_{*}
+ \vert\kern-0.25ex\vert\kern-0.25ex\vert \bar{y}_h(\omega)-y^h(\omega) \vert\kern-0.25ex\vert\kern-0.25ex\vert_{h},
\end{equation}
since $\bar{y}_h(\omega) -y^h(\omega) \in { V^h(\omega)}$. 
Subsequently, using the coercivity property in Lemma \ref{lem3.1} and \eqref{eq4.8}, we have
\begin{eqnarray}
\nonumber \vert\kern-0.25ex\vert\kern-0.25ex\vert \bar{y}_h(\omega)-y^h(\omega) \vert\kern-0.25ex\vert\kern-0.25ex\vert_{h}^2 
& \lesssim & A_h(\bar{y}_h(\omega) - y^h(\omega),\bar{y}_h(\omega) - y^h(\omega)) \\
\nonumber  & = & ({\bar{u}_h(\omega)-\bar{u}(\omega), \bar{y}_h(\omega) - y^h(\omega)})_{\mathcal{D}(\omega)} \\   
\nonumber  & \lesssim & \|\bar{u}_h(\omega)-\bar{u}(\omega)\|_{0,\mathcal{D}(\omega)}\, \|\bar{y}_h(\omega) - y^h(\omega)\|_{0,\mathcal{D}(\omega)} \\
& \lesssim & \|\bar{u}_h(\omega)-\bar{u}(\omega)\|_{0,\mathcal{D}(\omega)}\, 
\vert\kern-0.25ex\vert\kern-0.25ex\vert \bar{y}_h(\omega)-y^h(\omega) \vert\kern-0.25ex\vert\kern-0.25ex\vert_{h}   \label{eq4.16}
\end{eqnarray}
Combining \eqref{eq4.15} with \eqref{eq4.16}, estimate \eqref{eq4.6} follows.
The proof of estimate \eqref{eq4.7} is similar.
\end{proof}
Theorem \ref{errors1} along with the a--priori bounds for the approximate solutions  $y^h(\omega)$ and $p^h(\omega)$ to the auxiliary problems \eqref{DOCPa}-\eqref{DOCPb} from Lemma \ref{lemma3.2} lead to the following error estimates:
\begin{theorem}\label{errors2}
Let $(\bar{y}(\omega), \bar{p}(\omega), \bar{u}(\omega))\in H^2(\mathcal{D}(\omega))\times H^2(\mathcal{D}(\omega))\times L^2(\mathcal{D}(\omega))$
and $(\bar{y}_h(\omega), \bar{p}_h(\omega), \bar{u}_h(\omega))\in { V^h(\omega)} \times {V^h(\omega)}\times L^2(\mathcal{D}(\omega))$ be the
solutions of the continuous optimal control problem \eqref{OCPa}-\eqref{OCPc} 
and its discrete counterpart \eqref{dOCPa}-\eqref{dOCPc}, respectively, for some fixed realization $\omega\in\Omega$. Then, we obtain
\begin{align*}
\|\bar{y}(\omega)-\bar{y}_h(\omega)\|_{0,\mathcal{D}(\omega)}+\|\bar{p}(\omega)-\bar{p}_h(\omega)\|_{0,\mathcal{D}(\omega)}  + \|\bar{u}(\omega)-\bar{u}_h(\omega)\|_{0,\mathcal{D}(\omega)} \\
\lesssim h^2 \Theta(\bar{u}(\omega), { y_d}, C_\omega)
\\
\left\vert\kern-0.25ex\left\vert\kern-0.25ex\left\vert{\bar{y}(\omega)-\bar{y}_h(\omega)}
\right\vert\kern-0.25ex\right\vert\kern-0.25ex\right\vert_{*} + 
\left\vert\kern-0.25ex\left\vert\kern-0.25ex\left\vert{\bar{p}(\omega)-\bar{p}_h(\omega)}
\right\vert\kern-0.25ex\right\vert\kern-0.25ex\right\vert_{*}  \lesssim  h\, \Theta(\bar{u}(\omega), { y_d}, C_\omega),
\end{align*}
where $\Theta(\bar{u}(\omega), {y_d}, C_\omega)=\|\bar{u}(\omega)\|_{0,\mathcal{D}(\omega)} + \|{\ y_d}\|_{0,\mathcal{D}(\omega)} + C_{\omega}({ f,g_D,g_N})$.
\end{theorem}

 {We should remark that due to the definition  \eqref{star_norm} and the equations \eqref{OCPc}, \eqref{dOCPc}, Theorem \ref{errors2} readily yields the error estimates
\begin{align}
|\bar{y}(\omega)-\bar{y}_h(\omega)|_{1,\mathcal{D}(\omega)}+|\bar{p}(\omega)-\bar{p}_h(\omega)|_{1,\mathcal{D}(\omega)} 
& \lesssim
\left\vert\kern-0.25ex\left\vert\kern-0.25ex\left\vert{\bar{y}(\omega)-\bar{y}_h(\omega)}
\right\vert\kern-0.25ex\right\vert\kern-0.25ex\right\vert_{*} + 
\left\vert\kern-0.25ex\left\vert\kern-0.25ex\left\vert{\bar{p}(\omega)-\bar{p}_h(\omega)}
\right\vert\kern-0.25ex\right\vert\kern-0.25ex\right\vert_{*}  \lesssim  h\, \Theta(\bar{u}(\omega), { y_d}, C_\omega),
\label{L1-norm_state_error} \\
|\bar{u}(\omega)-\bar{u}_h(\omega)|_{1,\mathcal{D}(\omega)}
& \lesssim
\left\vert\kern-0.25ex\left\vert\kern-0.25ex\left\vert{\bar{u}(\omega)-\bar{u}_h(\omega)}
\right\vert\kern-0.25ex\right\vert\kern-0.25ex\right\vert_{*}  \lesssim  h\, \Theta(\bar{u}(\omega), { y_d}, C_\omega).
\label{L1-norm_control_error}
\end{align}
} 

Once the optimality system is discretized, a linear system of equations is obtained which will give the approximate solution to the mathematical model equations. For the non-deterministic approximation of the optimal solution in the probability domain $\Omega$ we will employ Quasi Monte Carlo quadrature to decrease the computational cost of a standard Monte Carlo sampling.  
\section{Monte Carlo simulations}\label{section5}
In robust optimal control problems with uncertainties we are interested in computing the first and second order statistical moments of our quantities of interest, that is, the expectation and the variance, {under the assumption that all uncertain parameters are independent and identically distributed (i.i.d.) uniforms in [0,1]}. Monte Carlo and Quasi Monte Carlo methods --between others-- can be used to approximate high dimensional integrals numerically and their basic aspects are illustrated in the next paragraphs.  
\subsection{Classical Monte Carlo simulation}
Monte Carlo method is based on the probabilistic interpretation of an integral that can be
expressed as the average or expectation of a real-valued function $h$ over the $s$-dimensional unit cube, namely, 
$I(h) =\mathbb{E}[h]= \int_{[0,1)^s}h(x)dx$. Then an empirical approximation to the expectation is given
by sample averaging  
\begin{equation}\label{average}
Q_N(h)=\frac{1}{N}\sum_{k=0}^{N-1}h(t_k),
\end{equation}
over the independent and uniformly distributed points $t_0, \ldots, t_{N-1}\in [0,1)^s$. A $2^8$-point set chosen randomly from the uniform distribution on $[0,1)^2$ is illustrated in Figure \ref{scatterplot}(a).  Apart from the ease of use, Monte Carlo method has the advantage of producing an unbiased estimate of the integral, i.e., $\mathbb{E}[Q_N(h)] = I(h)$ for any $N$. 
We recall that Monte Carlo integration converges with order of magnitude $\mathcal{O}(N^{-1/2})$ and although MC sampling comes as a natural choice, its very slow convergence leads to unaffordable computational cost when used in optimization algorithms that require iteration. This is the main motivation for switching to Quasi Monte Carlo methods to compute the statistical quantities involved in the current optimization problem.
\subsection{Quasi Monte Carlo methods}
Quasi Monte Carlo methods are a variant of ordinary Monte Carlo method that employ the same form (\ref{average}) in the approximation $\mathbb{E}[h]\approx Q_N(h)$, but guarantee small errors. The basic idea is to replace Monte Carlo's random samples by deterministic points selected in a way to cover $[0,1)^s$ more evenly. A deterministic QMC method with the same number of function evaluations achieves faster convergence --of order $\mathcal{O}(N^{-1})$-- than the classical MC method, thus, {provided some smooth dependence of the integrand on the parameters holds, QMC is superior than Monte Carlo.}

In order to obtain a family of well--distributed point sets in the unit cube  used for QMC integration, we choose \textit{rank-1 lattice rules}. The formula for the $k$-th point of an $N$-point rank-1 lattice rule with
generating vector $z\in \mathbb{Z}^s$ is given by
$\hat{t}_k = \left\{ \frac{kz}{N}\right\}=\left(\frac{kz}{N}\,\, \mathrm{mod}\,\, 1\right)$, $k=0, \ldots, N-1$,
in which the braces indicate that we take the fractional parts of each component in the vector. For a graphical demonstration of a two--dimensional lattice rule see Figure \ref{scatterplot}(b) and for further survey on QMC rules, see \cite{N1994,NC2006} and references therein.
A lattice rule can be turned into a randomized QMC method simply by shifting the lattice randomly, modulo one, with respect to each coordinate, that is
$
Q_N(h;\mathbf{\Delta})=\frac{1}{N}\sum_{k=0}^{N-1}h(\left\{\hat{t}_k+\mathbf{\Delta}\right\}),
$
where $\mathbf{\Delta}$ is a fixed random shift drawn from the uniform distribution on $[0,1)^s$. In this setting, we move all lattice points by the same amount, and we ``wrap" them back into the unit cube when necessary, to obtain shifted points in the unit cube. Figure \ref{scatterplot}(c) and Figure \ref{scatterplot}(d) illustrate the result of such a random shifting. To obtain uncorrelated estimators to $I(h)$ we take multiple independent random shifts $\mathbf{\Delta}_j\in [0,1)^s$, $j=0, \ldots, q-1$ for the same fixed lattice generating vector $z$ and we compute the average 
$Q_q(h)=\frac{1}{q}\sum_{j=0}^{q-1}Q_N(h;\mathbf{\Delta}_j)$.
In this case, the total number of evaluations of the integral is $qN$. Randomly shifted lattice rules maintain the fast convergence rate of a deterministic QMC method and provide an unbiased estimate to the integral approximation. Then a practical estimate of the Root--Mean--Square error of Quasi Monte Carlo method is calculated by
\begin{equation}\label{variance_shifted}
\textrm{Root-Mean-Square Error}\,\, {\approx} \,\, 
\Big(\frac{1}{{q-1}}\sum_{j=0}^{q-1}(Q_N(h;\mathbf{\Delta}_j)-Q_q(h))^2\Big)^{1/2}.
\end{equation}
\begin{figure}[h]
\centering
\begin{tabular}{cccc}
\subfigure[]{
\includegraphics[width=0.22\textwidth]{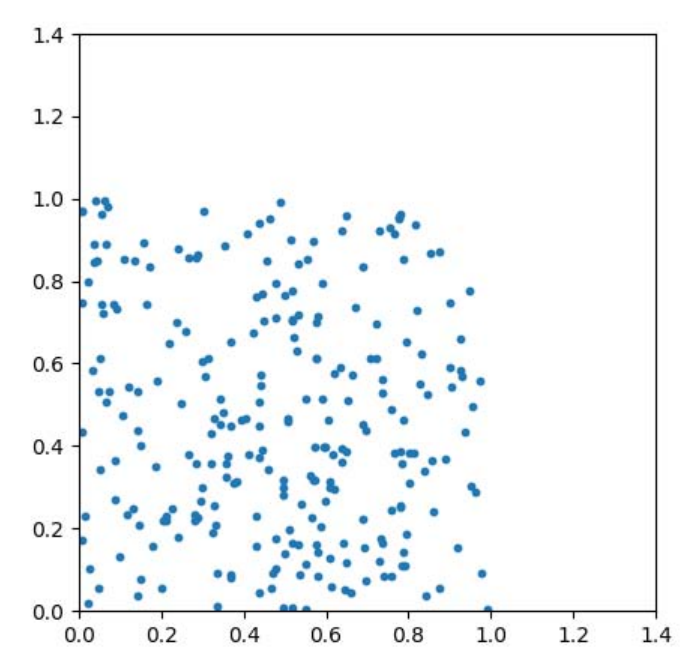}}\!\!\!\!\!
\subfigure[]{
\includegraphics[width=0.22\textwidth]{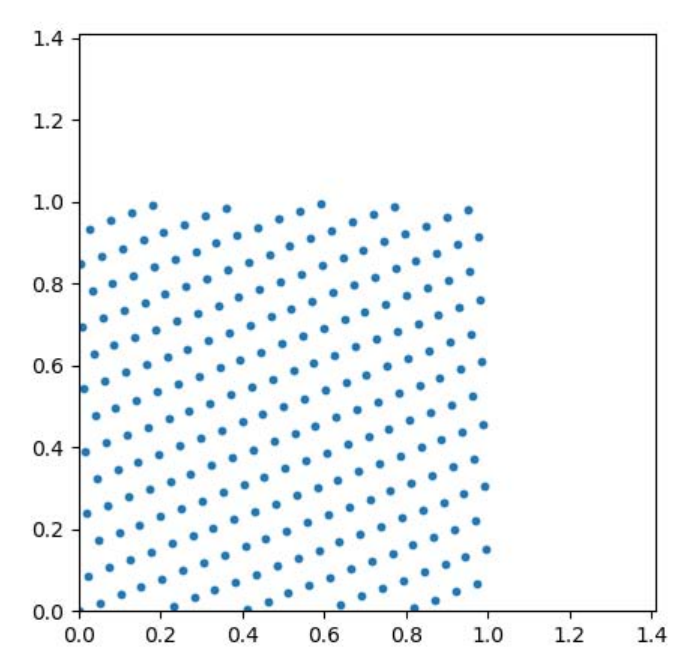}}\!\!\!\!\!
\subfigure[]{
\includegraphics[width=0.22\textwidth]{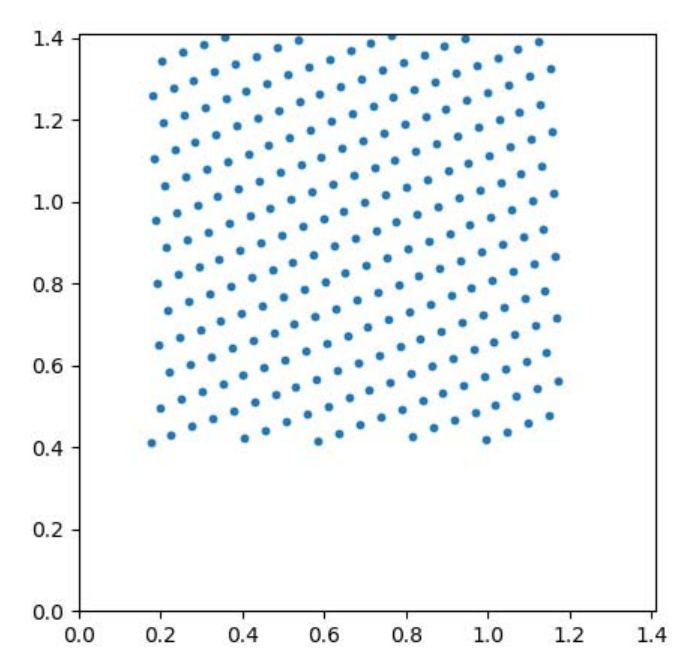}}\!\!\!\!\!
\subfigure[]{
\includegraphics[width=0.22\textwidth]{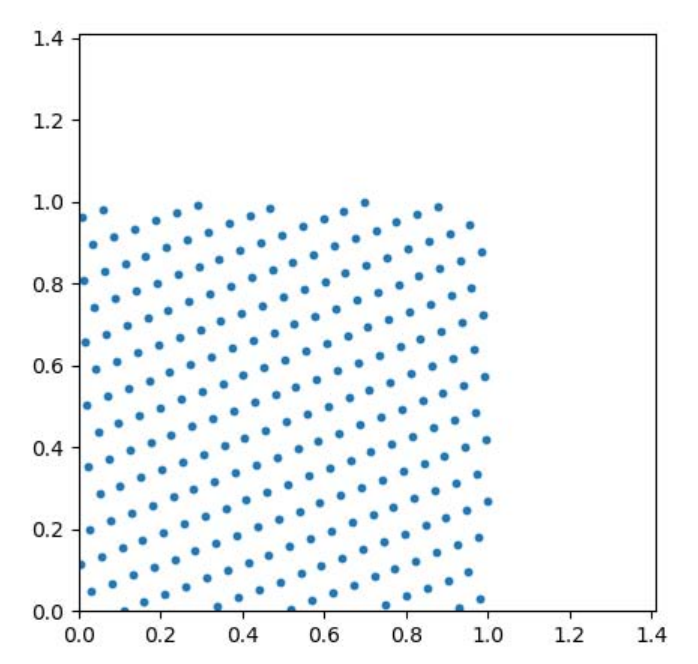}}
\end{tabular}
\caption{Illustration of a $2^{8}$-point set in the unit cube: (a) Monte Carlo method sampling randomly generated, (b) a lattice rule with generating vector $z=(1,127)$, (c) randomly-shifted lattice rule, (d) wrapped back inside the unit cube.}\label{scatterplot}
\end{figure}
In the following we describe an indirect method to approximate sequentially  the triple $(\bar{y}_h(\omega),\bar{p}_h(\omega),\bar{u}_h(\omega))$ in the physical domain $\mathcal{D}$ and we examine some preconditioning techniques. The latter are addressed by using iterative preconditioned Krylov subspace solvers on the deterministic state and adjoint equations while the computation of the deterministic control is achieved through a steepest descent method.
\section{Numerical solution and preconditioning techniques}
\label{section6}
It is well--known that the discrete optimality system in Lemma \ref{discrete} may equivalently be formulated as saddle--point problem. Indeed, denoting ${V^h(\omega)} =\textrm{span}\{\psi_1, \ldots, \psi_n\}$, where $\psi_j$, $j=1, \ldots, n$ are the finite element basis functions associated with the nodes of the triangulation $\mathcal{T}_h$, (\ref{dOCPa})-(\ref{dOCPc}) are rewritten in an algebraic level as a sparse linear system.
Hence taking the isomorphism $P_h: \mathbb{R}^{n}\to { V^h(\omega)}$ between the space of coefficients  and the space of finite element functions, the discrete triple $(\bar{y}_h(\omega), \bar{p}_h(\omega), \bar{u}_h(\omega))\in { V^h(\omega)}\times {V^h(\omega)}\times { V^h(\omega)}$ can be represented by
\begin{equation*}
\bar{y}_h(\omega) = 
P_h\mathbf{y}(\omega), \quad
\bar{p}_h(\omega) = 
P_h\mathbf{p}(\omega), \quad
\bar{u}_h(\omega) = 
P_h\mathbf{u}(\omega),
\end{equation*}
where $\mathbf{y}(\omega), \mathbf{p}(\omega), \mathbf{u}(\omega)\in\mathbb{R}^n$ are the vectors of coefficients. Inserting the above expressions into the system of equations (\ref{DOCP})
and letting $v_h = \psi_j$ for all $j=1, \ldots, n$, the discrete optimality system is described by the system
\begin{equation}\label{system1}
\begin{array}{ccccc}
K\mathbf{y}(\omega) & - & M\mathbf{u}(\omega) & = & \mathbf{d}(\omega) \\ 
K\mathbf{p}(\omega) & - & M\mathbf{y}(\omega) & = & \mathbf{b}(\omega) \\
\alpha M\mathbf{u}(\omega) & + & M\mathbf{p}(\omega) & = & 0, 
\end{array}
\end{equation}
\[
\text{where}\quad
M= \begin{bmatrix}
M_{ij}
\end{bmatrix}=
\begin{bmatrix}
({\psi_j, \psi_i})_{\mathcal{D}(\omega)}
\end{bmatrix}\in\mathbb{R}^{n\times n}
\,\,\,\textrm{and}\,\,\,
K = \begin{bmatrix}
K_{ij}
\end{bmatrix} = 
\begin{bmatrix}
A_h(\psi_j, \psi_i) 
\end{bmatrix}\in\mathbb{R}^{n\times n} 
\]
are the \emph{mass} matrix and  the \emph{stiffness} matrix corresponding to the CutFEM discretization,
respectively. Furthermore, the right--hand side vectors are given by
\[
\mathbf{b}(\omega)=\begin{bmatrix} -({{ y_d}, \psi_i})_{\mathcal{D}(\omega)} \end{bmatrix}\in\mathbb{R}^{n}
\,\,\,\textrm{and}\,\,\,
\mathbf{d}(\omega)=\begin{bmatrix} L_h(\psi_i) \end{bmatrix}\in\mathbb{R}^{n}.
\]
The system \eqref{system1} can be further reduced if we substitute the control by means of 
the last equation, i.e., $\mathbf{u}(\omega)=-\alpha^{-1}\mathbf{p}(\omega)$. Hence,
\begin{equation}\label{system2}
\mathcal{A}:=\begin{bmatrix}
K & \alpha^{-1}M \\
-M &  K 
\end{bmatrix}
\begin{bmatrix}
\mathbf{y}(\omega)\\ \mathbf{p}(\omega)
\end{bmatrix}=
\begin{bmatrix}
\mathbf{d}(\omega)\\ \mathbf{b}(\omega)
\end{bmatrix}.
\end{equation}
Systems of the latter type are typically very poorly conditioned, so classical
iterative solvers need to be coupled with a preconditioner, in order to achieve
fast convergence \cite{S2003}. Hence, we investigate the performance of preconditioning on \eqref{system2}, emphasizing on a block diagonal structure \cite{RDW2010}.
Let the block diagonal $(BD)$ matrix
$\mathcal{P}_{BD}=\mathrm{diag}(P_1,P_2)=\begin{bmatrix}
P_1 & 0 \\
0 & P_2 \\
\end{bmatrix}\in\mathbb{R}^{2n\times 2n},
$
where $P_{1}$, $P_{2}\in\mathbb{R}^{n\times n}$ are symmetric positive definite matrices which are
spectrally equivalent to the stiffness matrix $K\in\mathbb{R}^{n\times n}$. 
Applying $\mathcal{P}_{BD}^{-1}$ to the left of the matrix 
$\mathcal{A}$ in \eqref{system2}, we have
\[
\widehat{\mathcal{A}}=\mathcal{P}_{BD}^{-1}\mathcal{A}=
\begin{bmatrix}
P_1^{-1}K    & \alpha^{-1}P_1^{-1}M \\
-P_{2}^{-1}M & P_{2}^{-1}K \\
\end{bmatrix},
\]
which leads to the preconditioned linear system 
\begin{subequations}\label{Psystem}
\begin{eqnarray}
P_1^{-1}K\mathbf{y}(\omega)   & = & P_1^{-1}(M\mathbf{u}(\omega) + \mathbf{d}(\omega)) \label{Psystem:a}\\
P_{2}^{-1}K\mathbf{p}(\omega) & = & P_{2}^{-1}(M\mathbf{y}(\omega) + \mathbf{b}(\omega)). \label{Psystem:b}
\end{eqnarray}
\end{subequations}
In view of the spectral equivalence of $P_1$ and $P_2$, note
\[
\beta x_1^T P_1 x_1 \leq x_1^T K x_1 \leq \delta x_1^T P_1 x_1, \,\,\,\textrm{and}\,\,\,
\eta x_2^T P_{2} x_2 \leq x_2^T K x_2 \leq \theta x_2^T P_{2} x_2,
\]
for all $x_1, x_2\in \mathbb{R}^{n}\setminus\{0\}$ and some positive constants $\beta, \delta, \eta, \theta$. For $x=\begin{bmatrix}
x_1 \\ x_2 \end{bmatrix}$, it is readily verified that
\[
\min\{\beta,\eta\}  
\leq \frac{x^T \mathrm{diag}(K, K)x}{x^T\mathcal{P}_{BD} x} \leq 
\max\{\delta,\theta\}.
\]
The spectral condition number $\kappa(P_{BD}^{-1}\mathcal{A})$ is defined by the quotient of the largest and smallest eigenvalues of $P_{BD}^{-1}\mathcal{A}$ and it satisfies
\[
\kappa(P_{BD}^{-1}\mathcal{A})\leq \frac{\max\{\delta, \theta\}}{\min\{\eta,\theta\}}.
\]
Thus it is sufficient to find good preconditioners for the symmetric and positive definite stiffness matrix $K$, derived from the discretization scheme. In this case, well known Krylov subspace methods such as the Conjugate Gradient (CG) algorithm can be employed for the numerical implementation of \eqref{Psystem:a}, \eqref{Psystem:b}. 
It is also known that the condition number of $K$ is uniformly bounded with an upper bound 
of order $\mathcal{O}(h^{-2})$ independently of the location of the boundary relative to the background mesh (see \cite[Lemma 11]{BH2012}). 
This is due to ghost penalty stabilization term in cut finite elements that yields robust conditioning of the stiffness matrix with respect to the position of the boundary.

The size of the discretized optimality system  \eqref{Psystem:a}-\eqref{Psystem:b} is often too large and sparse for practical solution. Thus, instead of solving it directly as a single coupled system, an indirect algorithm will be used for its deterministic implementation, see Algorithm \ref{alg:seq1}. This algorithm is based on the gradient descent method to approximate the control variable that minimizes the cost functional, with descent direction given by the negative gradient $-\nabla J_h'$. It requires several evaluations -- iterations\footnote{Preconditioning techniques for such iterative methods have been extensively studied, see e.g. \cite{Trol_reaction10}, 
and references therein, and they will not be examined in the present work.} -- for the control with respect to an optimal step size parameter chosen in a selected search direction. Further, for every such evaluation one needs to solve sequentially the linear systems \eqref{Psystem:a} and  \eqref{Psystem:b} arising from the deterministic state and adjoint variational forms. The basic steps of the deterministic optimization algorithm are outlined as follows:

\begin{algorithm}
\caption{Optimization approach with preconditioned CG}\label{alg:seq1}
\begin{algorithmic}
\STATE Set the data $y_d$, $f$, $g$ for fixed parameter $\omega$ and give tolerance $\epsilon$.\\
\STATE Initialize iteration counter $k=0$, control $u_h=u_h^{(0)}$. \\
\WHILE{$k =1, 2, \ldots$ }
\STATE Solve (\ref{Psystem:a}) for state $y^{(k)}$ with given $u^{(k)}$  
using CG preconditioned by $P_1$. \\
\STATE Solve (\ref{Psystem:b}) for adjoint $p^{(k)}$ with computed $y^{(k)}$ 
using CG preconditioned by $P_2$. \\
\IF{$|J_h(u^{(k)})-J_h(u^{(k-1)})|/J_h(u^{(k)})\leq \epsilon$}
\STATE break
\ELSE
\STATE update $u_{h}^{(k+1)}=u_{h}^{(k)}-\tau \nabla J_h(u^{(k)})$, with appropriate step size parameter $\tau>0$ \\
\ENDIF
\ENDWHILE
\end{algorithmic}
\end{algorithm}

The question now posed is which preconditioner is optimal in the sense of the number of iterations required for convergence 
and what is the dependency on the mesh size. Some basic techniques of computational interest are examined in the following subsections.

\subsection{Jacobi and Symmetrized Gauss-Seidel}
Let $K=L+D+L^T$ be a matrix splitting of the stiffness matrix $K$, 
with $D$ its diagonal and $L$ its strictly lower triangular part. 
The Jacobi and Gauss-Seidel are two basic preconditioners, which 
are formed by the iteration matrices  based on stationary 
iterative methods, $K_J=D$ and $K_{GS}=D+L$, {respectively}.
However, $K_{GS}$ is not symmetric, which means it is not appropriate for the CG algorithm. 
Therefore, {combining} a forward sweep using the lower triangular component $D+L$
of $K$ {and} a backward sweep using the upper component $D+L^T$ of $K$,
the symmetric (and positive definite) Gauss-Seidel preconditioner 
$
K_{SGS}=(D+L)D^{-1}(D+L^T)
$
is obtained.
In section \ref{section7}, we test the performance of $K_J$ and $K_{SGS}$ using Algorithm \ref{alg:seq1} for the considered elliptic optimal control model problem solved by CutFEM, and we set $P_1=K_J$ in \eqref{Psystem:a} and $P_2=K_{SGS}$ in \eqref{Psystem:b}.  
Although, a combined use of both direct and iterative method can be competitive in elimi\-nating the ill-conditioning of large and sparse matrices. This is possible, e.g., by employing  multigrid preconditioning techniques.

\subsection{Unfitted mesh multigrid}
In this subsection we refer to geometric multigrid methods. A characteristic feature of a multilevel scheme is the generation of a mesh hierarchy with nested finite element basis functions starting from coarsest to finest levels. Initially a smoothing process is performed on the coarse level eliminating the error components of high frequency and then a prolongation operator is used to transfer information from the coarser to the finer grid. This procedure is repeated until a desired level of refinement will have been reached.

In \cite{LGR2018} the authors introduce a multigrid scheme especially designed for \emph{unfitted mesh} finite element discretizations and they investigate its performance as a solver, as well as, a preconditioner to an elliptic interface boundary value problem. Our goal is to adopt this unfitted mesh multigrid and to investigate its performance as a preconditioner to solve an optimal control problem discretized by a cut elements finite element method. 

We consider a sequence of grid sizes $\{h_{\ell} \}_{\ell\geq 0}$,  which generates a hierarchy of nested  {quasi-uniform} triangulations $\{\mathcal{T}_\ell \}_{\ell\geq0}$, from coarsest to finest meshes. This triangulation procedure, contains the original domain $\mathcal{D}(\omega)$ and the mesh is not fitted to the boundary $\Gamma(\omega)$. Since $\Gamma(\omega)$ is characterized by a zero level set function, it is approximated in a piecewise linear way, $\Gamma_{\ell}(\omega)$, with respect to the finite element basis in $\mathcal{T}_{\ell}$. Then it is evident that $\Gamma_{\ell-1}(\omega)\neq\Gamma_{\ell}(\omega)$. Nevertheless, if we consider the extended domains $\mathcal{D}_{\mathcal{T}_{\ell}}$ covered by the active part of the triangulations $\{K\in \mathcal{T}_{\ell}: K\cap\mathcal{D}(\omega)\neq\emptyset\}$, then the coarse level elements intersected by $\Gamma_{\ell}(\omega)$ are also intersected by $\Gamma_{\ell-1}(\omega)$.
This fact leads to a straightforward generalization of this unfitted mesh multigrid method to the finite element spaces
\begin{equation*}\label{cutspace}
V_\ell=\left \{  \upsilon \in C^0(\bar {\mathcal{D}}_{T_{\ell}})\,:\,\upsilon |_K  \in P^1(K), \, \forall K\in \mathcal{T}_\ell \right \}.
\end{equation*}
The underlying finite element spaces, form a hierarchy of  embedded spaces due to $v|_{\bar {\mathcal{D}}_{T_{\ell}}}\in V_{\ell}$ for any $v\in V_{\ell-1}$. Thus the prolongation operator is naturally induced in the spaces of coefficients 
\[
R: \mathbb{R}^{n_{\ell-1}}\xrightarrow{P_{\ell-1}} V_{\ell-1}\subset V_{\ell}\xrightarrow{P_{\ell}^{-1}}\mathbb{R}^{n_{\ell}},
\]
where $P_\ell: \mathbb{R}^{n_\ell}\to V_\ell$, $n_\ell=\dim V_\ell$ is the isomorphism between the spaces of coefficients  and the spaces of finite element functions. The stiffness matrix at the $(\ell-1)$-level is defined as $K_{\ell-1}=R^*K_{\ell}R$,
where $R^*$ is the canonical restriction.

In this unfitted mesh multigrid scheme an effective smoother is also addressed for unfitted mesh discretizations. The proposed smoothing procedure is called ``Gauss-Seidel with interface correction", since it is based on a standard Gauss-Seidel iteration on the entire computational domain accompanied by a local correction on the boundary zone (see Algorithm 1 in \cite{LGR2018}). In our fictitious domain method this local correction is applied on the unknowns corresponding to the vertices of cut elements.

For our deterministic experiments we conduct one multigrid V-cycle as an inner iteration inside CG method, see Algorithm \ref{alg:seq2}. This approach follows the ideas of Algorithm \ref{alg:seq1} adapted on the family of stiffness matrices associated to the linear system \eqref{Psystem:a}, \eqref{Psystem:b} for a sequence of nested grids.
\begin{algorithm}
\caption{Optimization approach with multigrid V-cycle  preconditioner}\label{alg:seq2}
\begin{algorithmic}
\STATE Set data $y_d$, $f, g$, for fixed $\omega$, levels of refinement $l$ and tolerance $\epsilon$. \\ Iteration counter $k=0$. \\
\WHILE{$k =1, \ldots$ }
\FOR{$\ell=0, 1, \ldots, l$} 
\IF{$\ell=0$}
\STATE Define coarse FEM space $V_0$ and set $u_0^{(k)}=u_0^{(k-1)}\in V_0$.
\STATE Solve \eqref{Psystem:a}, \eqref{Psystem:b} using sparse direct solver. \\
\ELSE
\STATE Define finer FEM space $V_\ell$ and set $u_{\ell}^{(k)}=u_{\ell}^{(k-1)}\in V_\ell$.\\
\STATE Solve \eqref{Psystem:a}, \eqref{Psystem:b} using CG preconditioned by  multigrid V-cycle with one pre- and post-smoothing steps.
\ENDIF
\IF{$|J_\ell(u_{\ell}^{(k)})-J_\ell(u_{\ell}^{(k-1)})|/J_\ell(u_{\ell}^{(k)})\leq \epsilon$}
\STATE break
\ELSE
\STATE Update $u_{\ell}^{(k+1)}=u_{\ell}^{(k)}-\tau \nabla J_\ell(u^{(k)})$, with appropriate parameter $\tau>0$ \\
\ENDIF 
\ENDFOR 
\ENDWHILE
\end{algorithmic}
\end{algorithm}
%
\section{Examples}\label{section7}

We firstly validate the theoretical error estimates in Theorem \ref{errors2} using a deterministic numerical example with manufactured exact solutions to compute the errors for the state, the adjoint and the control approximations on a sequence of refined meshes. We also test the performance of the preconditioners discussed in section \ref{section6} in two and three-dimensional deterministic problems. To this end, we consider a fixed realization $\omega\in\Omega$ and we examine the problem formulated in \eqref{cost}-\eqref{mixedBVP} with \emph{Dirichlet} boundary conditions imposed on $\Gamma=\partial\mathcal{D}$, i.e. $g_D\equiv g$ and  $\Gamma_N=\emptyset$. Later on, we carry out a non-deterministic version of the optimal control problem to illustrate the computational efficiency of Quasi Monte Carlo method. In all cases we set the regularization parameter $\alpha=0.1$, the Nitsche stabilization $\gamma_D =10$ and penalty parameter $\gamma_1 = 0.1$. The experiments have been tested on Aris HPC system with Intel Xeon E5-4650v2 and 496 GB of RAM in a python environment using the open--source Netgen/NGSolve \cite{Scho14,ngsolve}, and ngsxfem \cite{LeHePreWa21,ngsxfem} finite element software.

\subsection{Deterministic case}
In each of the following examples, and in the spirit of the fictitious domain approach, the original spatial domain $\mathcal{D}$ is immersed into the domain $\widetilde{\mathcal{D}}$. 
We consider the spaces $V_\ell$ spanned by continuous, piecewise linear finite element basis functions on a sequence of regular, simplicial meshes in $\bar{\mathcal{D}}_{\mathcal{T}_{\ell}}$ obtained from an initial, regular triangulation of $\widetilde{\mathcal{D}}$ by recursive, uniform bisection of simplices. To find $\bar{u}_\ell$ that minimizes the cost functional $J_\ell(u_\ell)$ we use Algorithm \ref{alg:seq1} in which several evaluations of $y_\ell$, $p_\ell$ are performed. For each evaluation we solve \eqref{Psystem:a} for the state and \eqref{Psystem:b} for the adjoint  with CG method accelerated by three types of preconditioners, as described in section \ref{section6}.
Subsequently, we demonstrate the spectral condition numbers of the corresponding preconditioned matrices and the number of iterations until a residual error of order $10^{-8}$ will have been reached. The first example examines and confirms the convergence rates of the CutFEM method, computed by the ratio
$
\mathrm{EOC}=\frac{\log(\|e^{(\ell-1)}\|)-\log(\|e^{(\ell)}\|)}{\log(2)},
$
where $\|e^{(\ell)}\|$ denotes the error at level $\ell$ between the exact optimal solution and the cut finite elements approximation with respect to a given norm. 

\begin{example}\label{ex1}
Let $\widetilde{\mathcal{D}}=[-1.5, 1.5]^2$ and the interface described by the unit circle. For the numerical evaluation of errors, 
we choose the manufactured optimal solution 
\begin{equation*}
(\bar{y}, \bar{p}, \bar{u})=
(\sin(0.5\pi x_1)\sin(0.5\pi x_2),
-0.1 (x_1^2+x_2^2-1)\sin(0.5\pi x_1),
(x_1^2+x_2^2-1)\sin(0.5\pi x_1)),
\end{equation*}
with desired and source functions given by
\begin{align*}
y_d(x_1,x_2) & =  
0.025\Big[\big(\pi^2(x_1^2+x_2^2-1)-16\big)\sin(0.5\pi x_1)
-8\pi x_1\cos(0.5\pi x_1)\Big] + \sin(0.5\pi x_1) \sin(0.5\pi x_2),
\\ 
f(x_1,x_2) & =  0.5\pi^2\sin(0.5\pi x_1) \sin(0.5\pi x_2) - 
(x_1^2+x_2^2-1)\sin(0.5\pi x_1),
\end{align*}
respectively, such that they verify the control system 
with non homogeneous Dirichlet boundary conditions. 

The errors with respect to the $L^2(\mathcal{D})$  {norm} and $H^1(\mathcal{D})$  {semi-norm} for the domain $\mathcal{D}$ computed by a direct sparse solver are visualized in Table \ref{table1} and Table \ref{table2}, respectively. 
Obviously, the observed rates adhere nicely with the theoretical optimal way  that the finite elements method converges, as proved in Theorem \ref{errors2}  {and error estimates \eqref{L1-norm_state_error}, \eqref{L1-norm_control_error}}. In particular, they confirm the second order and the first order accuracy with respect to the $L^2(\mathcal{D})$  {norm} and $H^1(\mathcal{D})$  {semi-norm}, respectively.

\begin{table}[htbp]
\centering
\begin{tabular}{ccccccc}
\toprule
$h_{\max}$ & $\|\bar{y}-\bar{y}_h\|_{0,\mathcal{D}}$ & EOC & $\|\bar{p}-\bar{p}_h\|_{0,\mathcal{D}}$ & EOC & $\|\bar{u}-\bar{u}_h\|_{0,\mathcal{D}}$ & EOC  \\
\midrule
$2^{-2}$ & 0.914502e-2 &      & 0.289888e-2 &      & 0.289888e-1 &   \\
$2^{-3}$ & 0.263861e-2 & 1.79 & 0.073717e-2 & 1.98 & 0.073717e-1 & 1.98 \\
$2^{-4}$ & 0.057354e-2 & 2.20 & 0.017087e-2 & 2.11 & 0.017087e-1 & 2.11 \\
$2^{-5}$ & 0.013884e-2 & 2.05 & 0.004151e-2 & 2.04 & 0.004151e-1 & 2.04 \\
$2^{-6}$ & 0.003512e-2 & 1.98 & 0.001015e-2 & 2.03 & 0.001015e-1 & 2.03 \\ 
$2^{-7}$ & 0.000881e-2 & 2.00 & 0.000250e-2 & 2.02 & 0.000250e-1 & 2.02 \\
$2^{-8}$ & 0.000217e-2 & 2.02 & 0.000062e-2 & 2.01 & 0.000062e-1 & 2.01 \\
\midrule
Mean     &             & 2.01 &            & 2.03 &            & 2.03 \\
\bottomrule
\end{tabular}
\medskip
\caption{Test \ref{ex1}: $L^2(\mathcal{D})$ errors and experimental order of convergence for the computed state $\bar{y}_h$, adjoint state $\bar{p}_h$ and control $\bar{u}_h$ without preconditioning.}\label{table1}
\end{table}
\begin{table}[htbp]
\centering
\begin{tabular}{ccccccc}
\toprule
$h_{\max}$ & $|\bar{y}-\bar{y}_h |_{1,\mathcal{D}}$ & EOC & $|\bar{p}-\bar{p}_h |_{1,\mathcal{D}}$ & EOC 
& $|\bar{u}-\bar{u}_h |_{1,\mathcal{D}}$ & EOC  \\
\midrule
$2^{-2}$ & 0.271164 &      & 0.426922e-1 &      & 0.426922 &   \\
$2^{-3}$ & 0.145505 & 0.90 & 0.222987e-1 & 0.94 & 0.222986 & 0.94 \\
$2^{-4}$ & 0.069620 & 1.06 & 0.110154e-1 & 1.02 & 0.110154 & 1.02 \\
$2^{-5}$ & 0.034615 & 1.01 & 0.055082e-1 & 1.00 & 0.055082 & 1.00 \\
$2^{-6}$ & 0.017446 & 0.99 & 0.027457e-1 & 1.01 & 0.027457 & 1.01 \\
$2^{-7}$ & 0.008728 & 1.00 & 0.001371e-1 & 1.00 & 0.013709 & 1.00 \\
$2^{-8}$ & 0.004341 & 1.01 & 0.006848e-1 & 1.00 & 0.006849 & 1.00 \\
\midrule
Mean     &         & 1.00 &           & 1.00  &     & 1.00 \\
\bottomrule
\end{tabular}
\medskip
\caption{Test \ref{ex1}: $H^1(\mathcal{D})$ errors and experimental order of convergence for the computed state $\bar{y}_h$, adjoint state $\bar{p}_h$ and control $\bar{u}_h$ without  preconditioning.}\label{table2}
\end{table}

The effect of preconditioning in accordance with the grid size is reported in Table \ref{table3}. We start from a coarse level $\ell=0$ and we use a mesh bisection algorithm to generate a sequence of meshes. Then, on each level of refinement we compute the spectral condition numbers of the original stiffness matrix $K$  and later after diagonal scaling ($K_J$) and preconditioned by the symmetrized variant of Gauss-Seidel ($K_{SGS}$) and the multilevel scheme ($K_{MG}$), as described in section \ref{section6}. 
All types of preconditioners are applied on CG to solve the state and adjoint equations, respectively (see Algorithms \ref{alg:seq1} and \ref{alg:seq2}), which terminates once the residual norm has reached the value $10^{-8}$. One V-cycle of multigrid method is applied as a preconditioner inside CG with one step of Gauss-Seidel pre-smoothing and post-smoothing on the computational domain followed by a correction smoother for the degrees of freedom on the cut elements.  

As observed, CG is prohibitive without the effect of preconditioning, due to the extremely slow convergence at high levels of refinement where $\kappa(K)$ grows drastically. Although diagonal scaling makes the situation better, it still gives high iteration counts at finer resolutions. The symmetrized version of Gauss-Seidel performs better than Jacobi, but overall CG preconditioned with multigrid significantly prevails. The multilevel scheme needs much fewer iterations to reach the desired solution accuracy independently of the levels. Moreover, at the finest level the spectral condition number of the preconditioned system matrix is close to 2.2 and the number of CG iterations are nearly constant. Therefore, using a multigrid method as a preconditioner inside CG method can be advantageous and sometimes more convenient than acting
as a solver. 

\begin{table}[htbp]
\centering
\resizebox{12cm}{!}
{
\medskip
\begin{tabular}{c|c|cc|cc|cc|cc}
	\toprule
	$\ell$  & dofs & $\kappa(K)$ & iter & $\kappa(K_J^{-1}K)$ & iter
	& $\kappa(K_{SGS}^{-1}K)$ & iter & $\kappa(K_{MG}^{-1}K)$ & iter\\
	\hline\midrule
	1 & 1159 & 870.96  & 147 &  616.64& 122 & 124.94 & 61  & 1.89 & 11 \\ 
	2 & 4402 & 3248.77 & 283 & 2271.04& 236 & 508.82 & 123 & 2.15 & 11\\ 
	3 & 17104& 10404.11& 576 & 6160.87& 458 &1930.91 & 245 & 2.16 & 12\\ 
	4 & 67461& 22183.03&1097 &13960.18& 877 &5437.50 & 502 & 2.20 & 12 \\ 
	\bottomrule
\end{tabular}}
\medskip
\caption{Test \ref{ex1}: Spectral condition numbers for the stiffness matrix $K$  (third co\-lumn) and the corresponding preconditioned stiffness matrices  with Jacobi $K_J$ (fifth column), symmetrized Gauss-Seidel $K_{SGS}$ (seventh column) and multigrid V-cycle (ninth column)  for various levels $\ell$. CG termination counts are also provided until residual norm has reached the value $10^{-8}$.}\label{table3}
\end{table}
\end{example}

\begin{example}\label{ex2}
In this case we examine a three-dimensional domain extending the model problem of the previous  example. For this context, let the cube $\widetilde{\mathcal{D}}=[-1.1, 1.1]^3$ encompassing a 
spherical domain centered at the origin whose boundary is the unit sphere. The source term and the desired state are given by 
\begin{align*}
f(x_1,x_2,x_3) & =  0.75\sin(0.5\pi x_1)\sin(0.5\pi x_2)\sin(0.5\pi x_3)
+ (1-x_1^2-x_2^2-x_3^2)\sin(0.5\pi x_1),\\
y_d(x_1,x_2,x_3) & =  \sin(0.5\pi x_1)\sin(0.5\pi x_2)\sin(0.5\pi x_3) -
0.025(1-x_1^2-x_2^2-x_3^2)\sin(0.5\pi x_1) \\
& \qquad\qquad\qquad\qquad\qquad\qquad\qquad\qquad
-0.6\sin(0.5\pi x_1)-0.2\pi x_1\cos(0.5\pi x_1),
\end{align*}
respectively, with a representative computed optimal solution illustrated in Figure \ref{image2}.

\begin{figure}[h]
\centering
\begin{tabular}{ccc}
\subfigure[The computed state $y^h$]{
	\includegraphics[width=0.25\textwidth]{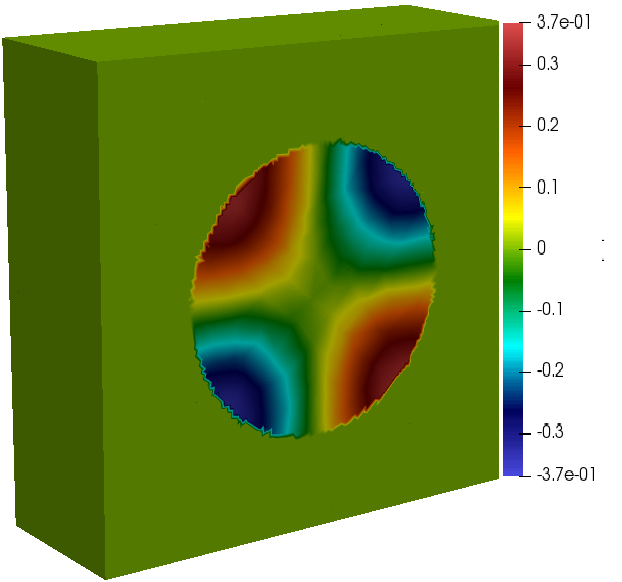}}\\ 
\subfigure[The computed adjoint $p^h$]{
	\includegraphics[width=0.25\textwidth]{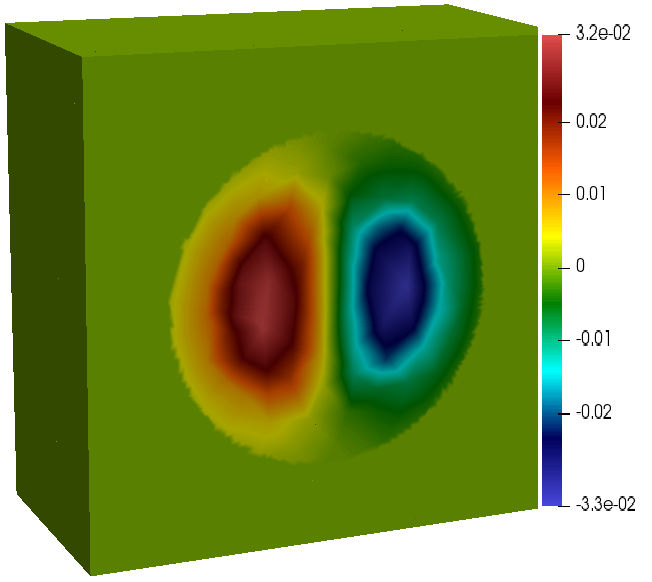}} \quad\quad
\subfigure[The computed control $u^h$]{
	\includegraphics[width=0.25\textwidth]{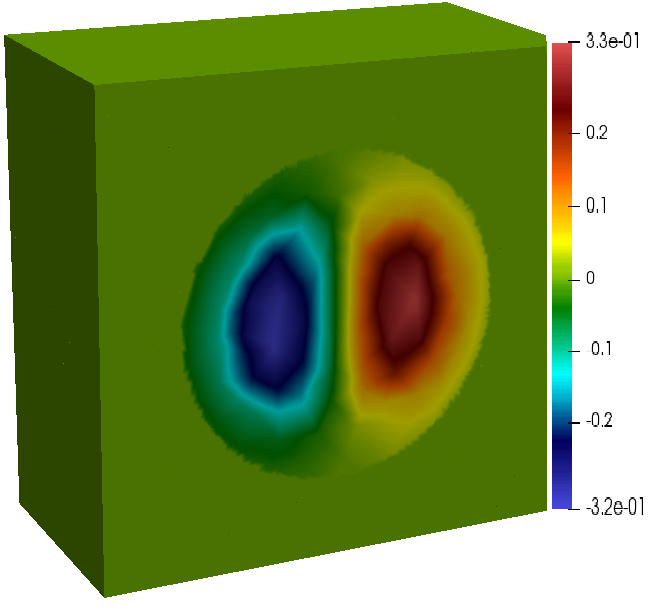}}
\end{tabular}
\caption{Test \ref{ex2}: The computed optimal solution triple for the three-dimensional control problem.}\label{image2}
\end{figure}
We present the current three-dimensional example aiming to compare the performance  and the condition numbers of the three different preconditioners in progressively refined computational meshes. Table \ref{table5} displays the number of CG iterations until the residual norm has reached the value  $10^{-8}$ and the condition numbers of the corresponding preconditioned matrices. Multigrid iterations consist of one V-cycle with one pre-smoothing and one post-smoothing steps of Gauss-Seidel smoother followed by a local interface correction. 

\begin{table}[ht]
\centering
\resizebox{12.5cm}{!}
{\medskip
\begin{tabular}{c|c|cc|cc|cc|cc}
	\toprule
	$\ell$ & dofs & $\kappa(K)$ & iter(CG) & $\kappa(K_J^{-1}K)$ & iter 
	& $\kappa(K_{SGS}^{-1}K)$ & iter & $\kappa(K_{MG}^{-1}K)$ & iter\\
	\hline\midrule
	1 &  912 &   2053.09 &  476 &  507.32 & 282 & 226.42 & 120 &  2.30 & 9 \\ 
	2 & 5616 &  15927.46 & 3174 &   873.10 & 727 & 713.92 & 278 &  5.86 & 15 \\ 
	3 &38624 & 273169.16 & 32407& 3698.96 &2849 &4414.50 & 976 & 19.88 & 25 \\ 
	\bottomrule
\end{tabular}}
\medskip
\caption{{Test \ref{ex2}:  Spectral condition numbers for the stiffness matrix $K$ and the corresponding preconditioned stiffness matrices, with Jacobi $K_J$ (fifth column), symmetrized Gauss-Seidel $K_{SGS}$ (seventh column) and multigrid V-cycle (ninth column)  for various levels $\ell$. CG iterations are  also provided which terminate after the residual norm has reached the value  $10^{-8}$.}}\label{table5}
\end{table}
\end{example}
\begin{example}\label{ex3}
This example serves to illustrate the behavior of the preconditioned matrix associated to the cut elements discretization in a more complicated and more realistic geometry e.g. of exhaust gaskets used in motorcycles or housing appliances.  The geometry is described by the level set 
\begin{equation}\label{complex_geometry}
\begin{split}
\phi((x_1,x_2),(\omega_1,\omega_2)) = (x_1^2+x_2^2 - 1) ((x_1-1.5)^2+x_2^2-0.02)((x_1+1.5)^2+x_2^2-0.02) \\
\cdot((4/9)x_1^2+0.0625x_2^2-(1/\omega_1)
-\omega_2\cos(\textrm{arctan}(5x_2/x_1))),
\end{split}
\end{equation}
with $(\omega_1,\omega_2)=(9, 2)$, prescribing the boundary of the red-colored area illustrated in Figure \ref{cutfig}, and with additional difficulty of no Lipschitz geometry onto some points . This geometry is embedded in the background domain $\tilde{\mathcal{D}}=[-3, 3]\times [-2.5, 2.5]$ and in order to make a visual comparison we also give an approximation of the optimal state, adjoint and control solutions in Figure \ref{complex_fig} regarding this more complex geometry.

The efficiency of the preconditioners in this example is conspicuous, see Table \ref{table4}. The large condition number of the initial stiffness matrix, in the  no preconditioning case, in a glance  shows  the ill-conditioning of the cut elements discretization system matrix,  and the necessity for preconditioning. Nevertheless, diagonal scaling does not give any significant improvement on the conditioning comparing with the symmetrized Gauss-Seidel or  with the multigrid preconditioner.

\begin{table}[htbp]
\centering
\resizebox{12.5cm}{!}
{\medskip
\begin{tabular}{c|c|cc|cc|cc|cc}
	\toprule
	$\ell$ & dofs & $\kappa(K)$ & iter(CG) & $\kappa(K_J^{-1}K)$ & iter 
	& $\kappa(K_{SGS}^{-1}K)$ & iter & $\kappa(K_{MG}^{-1}K)$ & iter \\
	\hline\midrule
	1 & 751  & 524.08 & 193 & 616.63 &  79 & 114.17 & 38 & 1.48 & 9  \\ 
	2 & 9637 & 963.63 & 227 & 510.70 & 179 & 106.81 & 97 & 2.30 & 12 \\ 
	3 & 36723&3645.44 & 493 &2274.43 & 351 & 412.72 &201 & 2.32 & 12 \\ 
	\bottomrule
\end{tabular}}
\medskip
\caption{{Test \ref{ex3}: Spectral condition numbers for the stiffness matrix $K$ and the corresponding preconditioned stiffness matrices, with Jacobi $K_J$ (fifth column), symmetrized Gauss-Seidel $K_{SGS}$ (seventh column) and multigrid V-cycle (ninth column)  for various levels $\ell$. CG iterations are  also provided which terminate after the residual norm has reached the value $10^{-8}$.}}\label{table4}
\end{table}
\begin{figure}[h]
\begin{center}
\begin{tabular}{cc}
	\subfigure[Computed state $y^h$]{
		\includegraphics[width=0.28\textwidth]{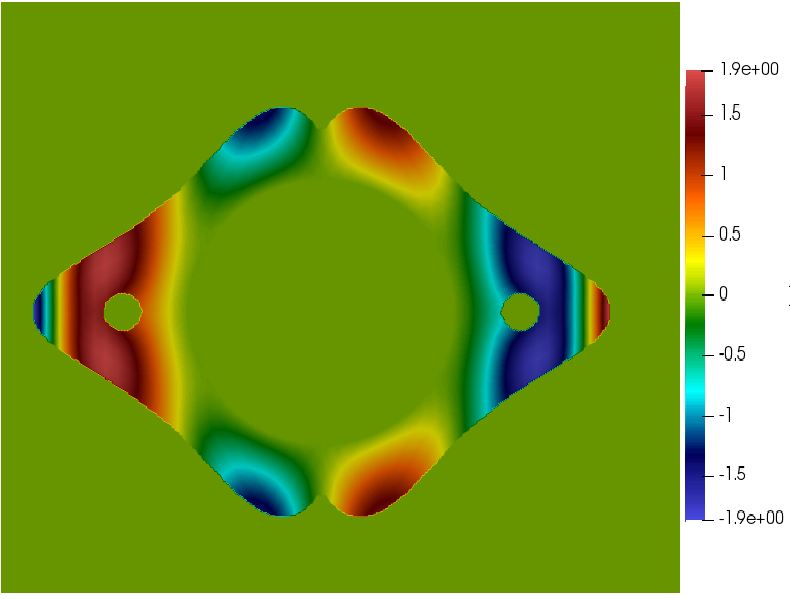}}\\ 
	\subfigure[Computed adjoint $p^h$]{
		\includegraphics[width=0.28\textwidth]{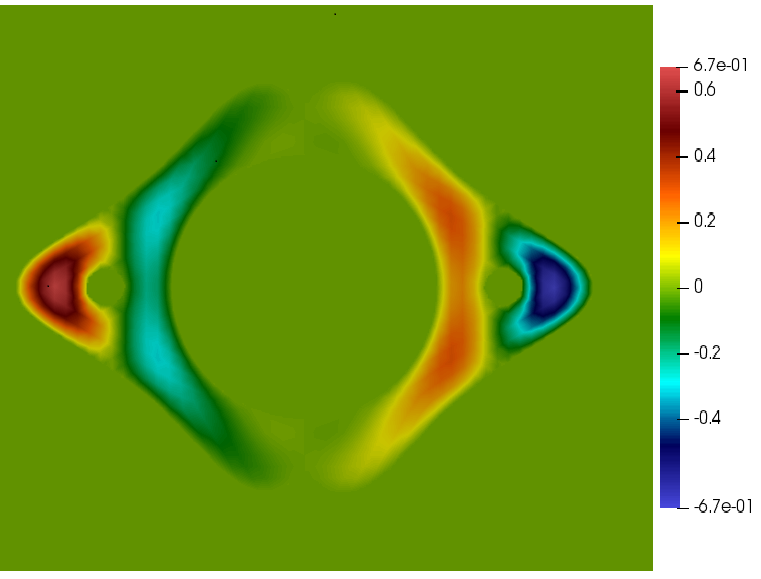}}\qquad
	\subfigure[Computed control $u^h$]{
		\includegraphics[width=0.28\textwidth]{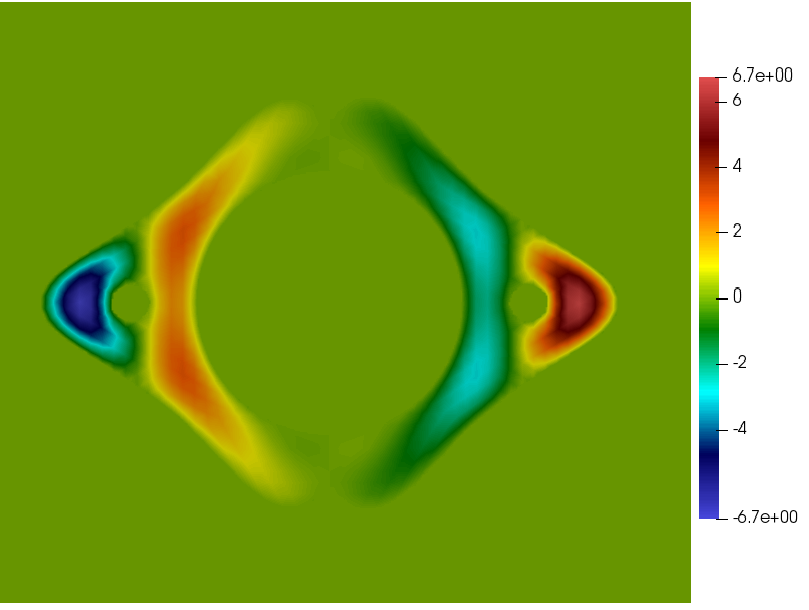}}
\end{tabular}
\caption{Test \ref{ex3}: The state, adjoint state and control variables approximation  with respect to a more complicated level set geometry.}\label{complex_fig}
\end{center}
\end{figure}
\end{example}
\subsection{Non-deterministic case}
Herein we exploit the sensitivity of the optimal control problem in the presence of random fluctuations in the domain field. Hence we examine the problem of finding the optimal solution pair of state $\bar{y}(\omega)\in H^1(\mathcal{D}(\omega))$ and control $\bar{u}(\omega)\in L^2(\mathcal{D}(\omega))$ such that 
\[
\min\mathcal{J}(y, u):=\frac{1}{2}
\|y(\omega_{i})-{y_d}\|_{0,\mathcal{D}(\omega_i)}^2 + \frac{\alpha}{2} \|u(\omega_{i})\|_{0,\mathcal{D}(\omega_i)}^2,
\]
subject to \,\,\,
$
\nonumber - \Delta y(\omega_i) = { f} + u(\omega_i) \,\,\,\textrm{in}\,\,\, \mathcal{D}(\omega_i), \,\,\,\textrm{and}\,\,\, 
y(\omega_i) = {g}  \,\,\,\textrm{on}\,\,\,\, \Gamma(\omega_i),
$
\\
for $i=1, \ldots, N$. As $\mathcal{D}(\omega)$ we take the parameterized geometry with boundary given by the level set \eqref{complex_geometry}.
For a representation of the geometrical deformations with respect to three samples see the uncovered areas illustrated at the first three pictures (from left to right) in Figure \ref{parametrization}. A comparison of all geometries together is visualized at the rightmost picture.

We again discretize the spatial domain $\mathcal{D}=[-3,3]\times [-2.5,2.5]$ by the cut finite elements method,  for mesh size $h = 0.075$, and a QMC quadrature rule to approximate the statistics of our quantities of interest over the parameter space $\Omega = [9,12] \times [2,3]$. As a first test, in order to investigate the efficiency of a QMC method, we implement it in its pure deterministic form by taking $N = 2^{m}$, $m=1, \ldots, 15$, two-dimensional lattice rules with generating vector $z=(1,127)$. For an illustration of a $2^{8}$-point rank-1 lattice rule see Figure \ref{scatterplot}(b).

To minimize the discrete cost functional, we use the lattice rules to sample the points in the parameter domain $\Omega$ and then we take them as inputs for the state and adjoint variational forms to solve the resulting systems. We execute Algorithm \ref{alg:seq1}, starting  with $u_h^{(0)}(\omega) = 1$ as the initial guess of the control variable, and we apply CG with a two-grid preconditioner. Then we ensemble the average and variance of our output quantities of interest, regarded to be the difference between the optimal state $\bar{y}_h(\omega)$ and the target state ${ y_d}$, the optimal state $\bar{y}_h(\omega)$, the optimal control $\bar{u}_h(\omega)$, and the cost functional $J_h(\bar{y}_h,\bar{u}_h)$. A visualization of the approximate optimal state $\bar{y}_h(\omega)$ for a single realization $\omega=(12, 3)$ is shown in Figure \ref{random_solution}.

\begin{figure}[htbp]
\centering
\includegraphics[width=0.44\textwidth]{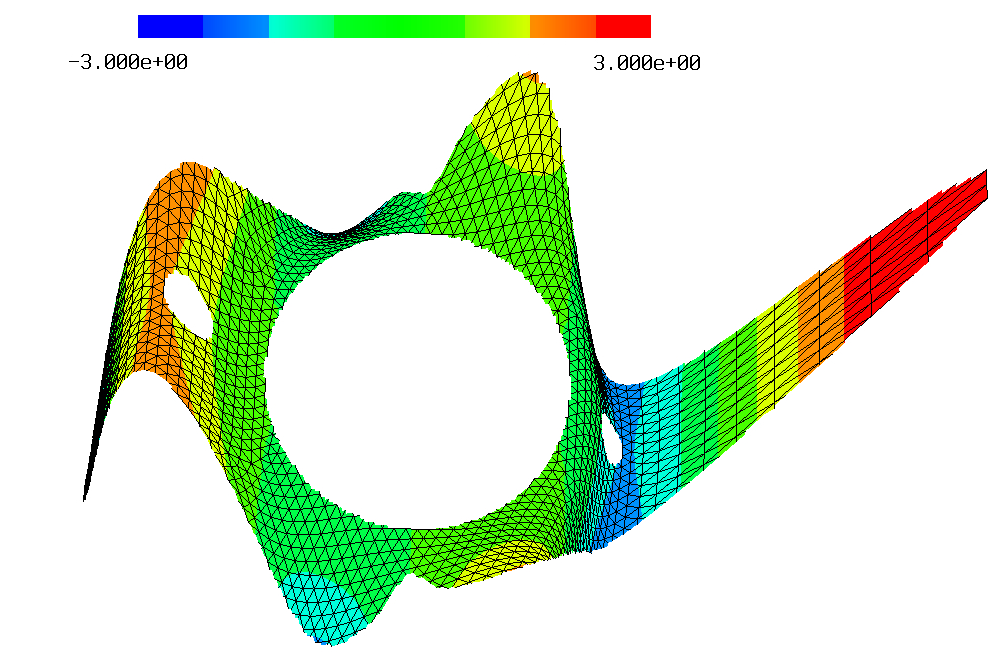}
\caption{Visualization of the computed optimal state $\bar{y}_h(\omega)$ in the parametrized geometry prescribed by \eqref{complex_geometry} for a single realization $\omega=(\omega_1,\omega_2)=(12,3)$.}\label{random_solution}
\end{figure}

The corresponding log-log plots of the errors are demonstrated in Figure \ref{QMC_rates}. The application of the Quasi Monte Carlo simulation in our model problem provides evi\-dence of its fast convergence properties achieving a rate of $\mathcal{O}(N^{-1})$ compared to Monte Carlo's rate of $\mathcal{O}(N^{-1/2})$.

\begin{figure}[htbp]
\centering
\includegraphics[width=0.73\textwidth]{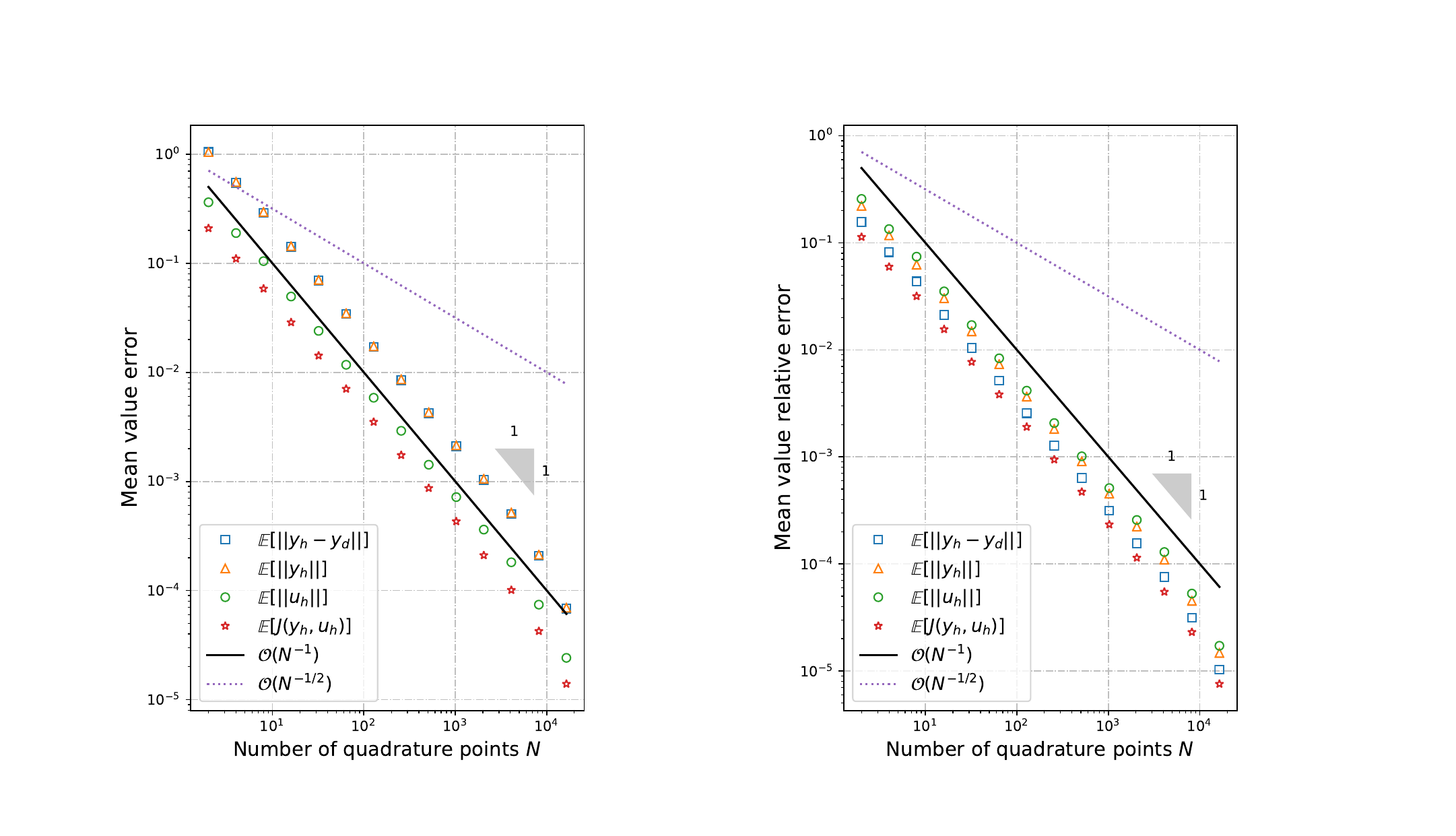}\\
\quad
\includegraphics[width=0.74\textwidth]{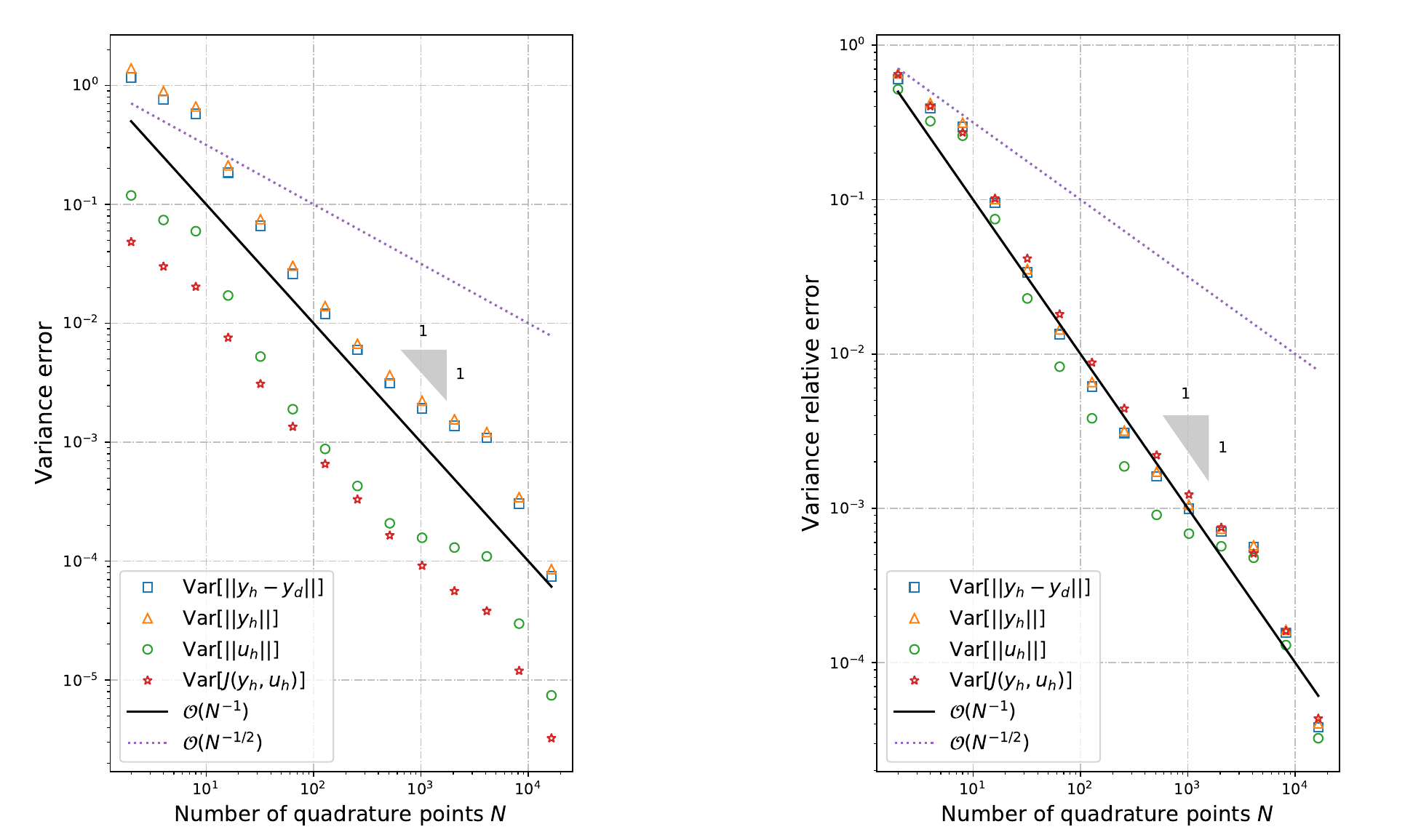}
\caption{Quasi Monte Carlo simulations with $N=2^{15}$ deterministic lattice points. The top pictures show the convergence rate of the standard error (left) and relative error (right) for the approximation of the integrals $\mathbb{E}[\|\bar{y}_h-y_d\|]$ (blue squares), $\mathbb{E}[\|\bar{y}_h\|]$ (orange triangles), $\mathbb{E}[\|\bar{u}_h\|]$ (green bullets) and $\mathbb{E}[J(\bar{y}_h, \bar{u}_h)]$ (red stars), while the bottom pictures refer to the errors of $\mathrm{Var}[\|\bar{y}_h-y_d\|]$ (blue squares), $\mathrm{Var}[\|\bar{y}_h\|]$ (orange triangles), $\mathrm{Var}[\|\bar{u}_h\|]$ (green bullets), $\mathrm{Var}[J(\bar{y}_h, \bar{u}_h)]$ (red stars). For reference we also plot Monte Carlo order of convergence (dotted line).}\label{QMC_rates}
\end{figure}

We continue our analysis by investigating the sensitivity of our quantities of interest with respect to the parameter domain $\Omega=[9,9.25]\times[2,2.25]$. In this effort, we focus on a randomized QMC form constructing randomly shifted lattice rules, as described in section \ref{section5} (see also Figure \ref{scatterplot}) and we obtain a practical estimate of the error by taking formula (\ref{variance_shifted}). To avoid spoiling the good QMC convergence rate achieved previously, we select the number of random shifts $\mathbf{\Delta}_{j}$, $j=1, \ldots, q$ to be small and fixed, say $q=16$. Then we increase successively the number of points until the desired error threshold of $10^{-5}$ is satisfied. The efficiency of this randomized version of QMC method is demonstrated in Figure \ref{QMC_rates_random_shift} for an increasing number of quadrature points until $qN=2^{16}$.
\begin{figure}[htbp]
\centering
\includegraphics[width=0.46\textwidth]{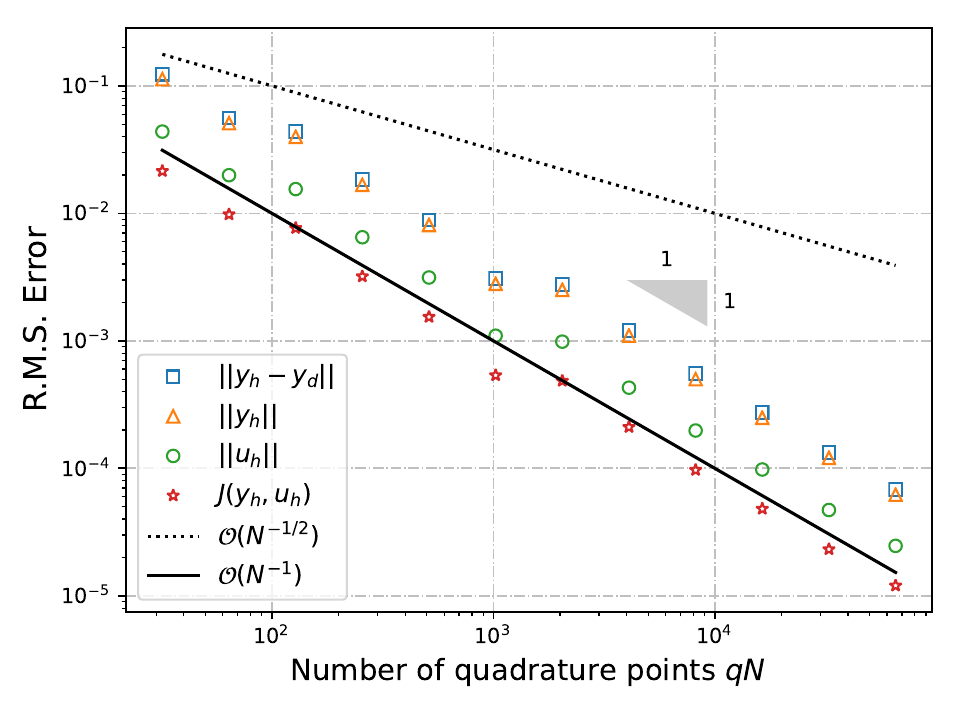}
\caption{Root-Mean-Square error for the quantities of interest $\|\bar{y}_h-y_d\|$ (blue squares), $\|\bar{y}_h\|$ (orange triangles), $\|\bar{u}_h\|$ (green bullets) and $J(\bar{y}_h, \bar{u}_h)$ (red stars). The pictures show the expected convergence rate of randomized Quasi Monte Carlo simulation with $qN=2^{16}$ randomly shifted lattice points. For reference we also plot Monte Carlo order of convergence (dotted line).}\label{QMC_rates_random_shift}
\end{figure}
\begin{remark}
We notice that in order to compute the mean value and variance errors and relative errors in the non-deterministic case, as ``true" values we use the mean and variance values of our quantities of interest as they are calculated at the highest level.
\end{remark}
\section{Conclusion}
We have considered an elliptic optimal control problem with uncertainties in the spatial domains. We discretized by a cut finite element method utilizing its value considering geometrical changes,  and in the probability domains we approximated  solutions with a randomized Quasi Monte Carlo method. We have also applied three types of basic preconditioners: Jacobi, symmetrized Gauss-Seidel and multigrid on a Conjugate Gradient iterative method in order to solve the linear systems arising from the discrete state and adjoint state variational forms. A comparison of the behavior of these preconditioners has been carried out on a set of two dimensional problems, firstly considering a domain which preserves the $H^2$--regularity property, then on a geometry with singularities (cusp points) and {finally}, a three-dimensional problem has been also considered. 

In the case of multigrid scheme the hierarchy of triangulations is easily obtained by taking the associated extended computational domains, with canonical prolongation--restriction operators and smoothing iterations in the domain and on the boundary zone. As expected, multigrid iteration as a preconditioner inside CG  method  leads  to  a  more  robust  conditioning of the stiffness matrix compared to the other two preconditioners, with a convergence factor independent of the level of the triangulation.

In addition, we have implemented Quasi Monte Carlo both in its deterministic and randomized forms concerning the sensitivity of the optimal control problem with respect to the uncertain parameters. Its effectiveness in variance reduction has been confirmed by the practical estimate of the Root--Mean--Square error. 

Finally,  in a future work, and based on  the good performance of the aforementioned techniques in optimal control with partial differential equations as constraints and unfitted mesh finite element methods, we are planning  to extend and to deal with preconditioning in optimal control constrained by time depended  systems, nonlinear partial differential equations, fluid flow phase systems, Stokes and Navier--Stokes flow approximations and other Quasi Monte Carlo variants for even better performance.

\section*{Acknowledgments}
This project has received funding from the Hellenic Foundation for Research and Innovation (HFRI) and  the  General  Secretariat  for  Research  and  Technology (GSRT),  
under  grant agreement No[1115] (PI: E. Karatzas). This work was supported by computational time granted from the National Infrastructures for Research and Technology S.A. (GRNET S.A.) in the National HPC facility - ARIS - under project ID pa190902 and the ``First Call for H.F.R.I. Research Projects to support Faculty members and Researchers and the procurement of high-cost research equipment'' grant 3270.



\begin{thebibliography}{00}
\bibitem{APSG2017}
A. Alexanderian, N. Petra, G. Stadler, and O. Ghattas, 
\emph{Mean-variance risk-averse optimal control of systems governed by PDEs with random parameter fields using quadratic approximations},
SIAM/ASA J. Uncert. Quant. 5(1) (2017), 1166-1192.

\bibitem{AUH2017}
A. Ali, E. Ullmann, and M. Hinze,
\emph{Multilevel Monte Carlo analysis for optimal control of elliptic PDEs with random coefficients}, SIAM/ASA J. Uncert. Quant. 5(1) (2017), 466-492.

\bibitem{BHP2019}
S. Badia, J. Hampton, and J. Principe,
\emph{Embedded Multilevel Monte Carlo for uncertainty quantification in random domains},
arXiv:1911.11965, 2019.

\bibitem{BV2019}
A. Van Barel, and S, Vandewalle,
\emph{Robust optimization of PDEs with random coefficients using a Multilevel Monte Carlo Method},
SIAM/ASA J. Uncertainty Quantification, 7(1) (2019), 174�202.

\bibitem{BDOS2016}
P. Benner, S. Dolgov, A. Onwunta, and M. Stoll,
\emph{Low-rank solvers for unsteady Stokes-Brinkman optimal control problem with random data},
Comput. Methods Appl. Mech. Engrg. 304 (2016), 26-54. 

\bibitem{BOS2015}
P. Benner, A. Onwunta, M. L. Stoll,
\emph{Low-rank solution of unsteady diffusion equations with stochastic coefficients}, 
SIAM/ASA J. Uncert. Quant. 3(1) (2015), 622-649. 

\bibitem{BT2018}
P. Benner, and C. Trautwein, 
\emph{A linear quadratic control problem for the stochastic heat equation driven by Q-Wiener processes}, J. Math. Anal. Appl. 457(1) (2018), 776-802.  

\bibitem{BT2019}
P. Benner, and C. Trautwein, 
\emph{Optimal Control Problems Constrained by the Stochastic Navier-Stokes 
Equations with Multiplicative L\'{e}vy Noise}, 
Mathematische Nachrichten 292(7) (2019), 1444-1461.  

\bibitem{BT}
P. Benner, and C. Trautwein, 
\emph{Optimal Distributed and Tangential Boundary Control for the Unsteady Stochastic Stokes Equations}, 
ESAIM: Control, Optimization and Calculus of Variations (to appear), 
arXiv:1809.00911v1, 2019.

\bibitem{BWB2019}
A. Bernland, E. Wadbro, and M. Berggren,
\emph{Shape Optimization of a Compression Driver Phase Plug},
SIAM J. Sci. Comput. 41(1) (2019), B181-B204.

\bibitem{BG14}
W. Bo, and J. W.\ Grove, 
\emph{A volume of fluid method based ghost fluid method for compressible multi-fluid flows}, 
Computers \& Fluids 90 (2014), 113-122.

\bibitem{BLSGUR2018}
Z. Bontinck, O. Lass, S. Sch\"{o}ps, H. De Gersem, S. Ulbrich and O. Rain, 
\emph{Robust optimisation formulations for the design of an electric machine}, 
IET Science, Measurement \& Technology, 12(8) (2018), 939-948.

\bibitem{BCHLM2014} 
E.\ Burman, S.\ Claus, P.\ Hansbo, M. G.\ Larson, and A.\ Massing, 
\emph{CutFEM: discretizing geometry and partial differential equations}, 
Int. J. Numer. Meth. Engrg.\ 104 (2014), 472-501.

\bibitem{BH2010} 
E.\ Burman, and P.\ Hansbo, 
\emph{Fictitious domain finite element methods using cut elements: I. 
A stabilized Lagrange multiplier method}, 
Comput. Methods Appl. Mech. Engrg.\ 199(41-44) (2010), 2680-2686.

\bibitem{BH2012} 
E.\ Burman, and P.\ Hansbo, 
\emph{Fictitious domain finite element methods using cut elements II. 
A stabilized Nitsche method}, 
Appl. Num. Math.\ 62(4) (2012), 328-341.

\bibitem{CCJ2017}
A. Capolei, L. H. Christiansen, and J. B. Jorgensen,
\emph{Offset risk minimization for open-loop optimal control of oil reservoirs},
IFAC PapersOnLine, 50(1) (2017), 10620-10625.

\bibitem{CR2003}
R. G. Carter, and H. H. Rachford, 
\emph{Optimizing line-pack management to hedge against future load uncertainty}, 
Pipeline Simulation Interest Group, Document ID PSIG-0306, PSIG Annual Meeting, 
2003, Bern, Switzerland.

\bibitem{CK2007}
C. Canuto, and T. Kozubek, 
\emph{A fictitious domain approach to the numerical solution of PDEs in stochastic domains}, 
Numer. Math. 107 (2007), 257-293.


\bibitem{Trol_reaction10}
E. Casas, C. Ryll, and F. Tr\"{o}ltzsch, 
\emph{Sparse optimal control of the Schl\"{o}gl and FitzHugh-Nagumo systems}, 
Comput. Methods Appl. Math., 13(4) (2013), 415-442.

\bibitem{CM2002}
E. Casas, and M. Mateos, 
\emph{Second order optimality conditions for semilinear
elliptic control problems with finitely many state constraints}, 
SIAM J. Control Optim. 40(5) (2002), 1431-1454.

\bibitem{CT2002}
E. Casas, and F. Tr\"{o}ltzsch, 
\emph{Second order necessary and sufficient optimality conditions for 
optimization problems and applications to control theory}, 
SIAM J. Optim. 13(2) (2002), 406-431.

\bibitem{Chiu13}
S. N. Chiu, D. Stoyan, W. S. Kendall, J. Mecke, 
\emph{Stochastic geometry and its applications}, 
Wiley Series in Probability and Statistics Book Series, 2013.

\bibitem{CK2014}
K. Chrysafinos, and E. N. Karatzas,
\emph{Error estimates for discontinuous Galerkin time-stepping schemes for 
Robin boundary control problems constrained to parabolic PDEs},
SIAM J. Numer. Anal.\ 52(6) (2014), 2837-2862.

\bibitem{DDG2002}
J. Deang, Q. Du, and M. D. Gunzburger,
\emph{Modeling and computation of random thermal fluctuations and material defects in the Ginzburg-Landau model for superconductivity},
J. Comput. Phys. 181 (2002), 45-67. 

\bibitem{DKS2013}
J. Dick, F. Y. Kuo, and I. H. Sloan,
\emph{High-dimensional integration: The quasi-Monte Carlo way},
Acta Numerica, 22 (2013), 133-288.

\bibitem{DL19}
M. Duprez and A. Lozinski, \emph{$\phi$--FEM: a finite element method on domains defined by level--sets},
arXiv: 1901.03966v3, 2019.

\bibitem{F1996}
G. Fishman, 
Monte Carlo: Concepts, Algorithms, and Applications, 
\emph{Springer Series in Operations Research and Financial Engineering}, Springer, 1996.

\bibitem{Gri85}
P. Grisvard, \emph{Elliptic problems in nonsmooth domains}, Monographs and Studies in Mathematics, vol. 24, Pitman, Massachusetts, 1985.

\bibitem{GLL2011}
M. D. Gunzburger, H.-C. Lee, and J. Lee, 
\emph{Error estimates of stochastic optimal Neumann boundary control problems}, 
SIAM J. Numer. Anal., 49(4) (2011), 1532-1552.

\bibitem{GWZ2012}
M. D. Gunzburger, C. G. Webster, and G. Zhang,
\emph{An adaptive wavelet Stochastic Collocation Method for irregular solutions of Partial Differential Equations with random input data}, In: Garcke J., Pfl\"{u}ger D. (eds) Sparse Grids and Applications--Munich 2012. Lecture Notes in Computational Science and Engineering, vol 97. Springer, Cham.

\bibitem{GWZ2014}
M. D. Gunzburger, C. G. Webster, and G. Zhang,
\emph{Stochastic finite element methods for partial differential equations with random input data}, 
Acta Numerica (2014), 521-650.

\bibitem{GKKSS2019}
P. A. Guth, V. Kaarnioja, F. Y. Kuo, C. Schillings, and I. H. Sloan,
\emph{A quasi-Monte Carlo method for an optimal control problem under uncertainty}, 
arXiv: 1910.10022v1, 2019.

\bibitem{HH2002} 
A.\ Hansbo, and P.\ Hansbo, 
\emph{An unfitted finite element method, based on Nitsche's method, 
for elliptic interface problems}, 
Comput. Methods Appl. Mech. Engrg., 191(47-48) (2002), 5537-5552.

\bibitem{HPS2016}
H. Harbrecht, M. Peters, and M. Siebenmorgen, 
\emph{Analysis of the domain mapping method for elliptic diffusion problems on random domains}, 
Numer. Math. 134 (2016), 823-856.

\bibitem{HS2010} 
R.\ Herzog, and E.\ Sachs, 
\emph{Preconditioned conjugate gradient method for optimal control problems 
with control and state constraints}, 
SIAM Journal on Matrix Analysis and Applications 31(5) (2010), 2291-2317.

\bibitem{H2005}
M.\ Hinze, 
\emph{A variational discretization concept in control constrained optimization:
The linear-quadratic case},
Comput. Optim. Applic., 30 (2005), 45-61. 

\bibitem{HPUU2008}
M.\ Hinze, R.\ Pinnau, M.\ Ulbrich and S.\ Ulbrich,
\emph{Optimization with PDE constraints}, 
vol. 23, Springer, 2008.

\bibitem{JGTW2015}
P. Jantsch, M. D. Gunzburger, A. Teckentrup, and C. Webster,
\emph{A multilevel stochastic collocation method for partial differential equations with random input data},
SIAM/ASA J. Uncert. Quant. 3 (2015), 1046-1074. 

\bibitem{KBR2019}
E. N.\ Karatzas, F.\ Ballarin, and G.\ Rozza, 
\emph{Projection-based reduced order models for a cut finite element method 
in parametrized domains}, Computers \& Mathematics with Applications 79(3) (2020), 833-851.

\bibitem{KSASR2019}
E. N.\ Karatzas, G.\ Stabile, N. Atallah, G.\ Scovazzi, and G.\ Rozza,
\emph{A  reduced order approach for the embedded shifted boundary FEM and a heat exchange system
on parametrized geometries}, In: Fehr J., Haasdonk B. (eds)
IUTAM Symposium on Model Order Reduction of Coupled Systems, Stuttgart, Germany, May 22-25, 2018. IUTAM Bookseries, vol 36. Springer, Cham (2020)

\bibitem{KSNSRa2019}
E. N.\ Karatzas, G.\ Stabile, L.\ Nouveau, G.\ Scovazzi, and G.\ Rozza,
\emph{A  reduced basis approach for PDEs  on parametrized geometries
based on the shifted boundary finite element method and application to
a Stokes flow},
Comput. Methods Appl. Mech. Engrg.\ 347 (2019), 568-587.

\bibitem{KSNSRb2019}
E. N.\ Karatzas, G.\ Stabile, L.\ Nouveau, G.\ Scovazzi, and G.\ Rozza,
\emph{A  reduced-order shifted boundary method for parametrized incompressible Navier-Stokes equations},s, Computer Methods in Applied Mechanics and Engineering {\bf{370}} (2020), 113-273.

%
\bibitem{KB18}
E. M.\ Kolahdouz, A. P. S.\ Bhalla, B. A.\ Craven, and B. E.\ Griffith, 
\emph{An immersed interface method for faceted surfaces}, 
arXiv:1812.06840v2, 2018.

\bibitem{KS2016}
A. Kunoth, and C. Schwab,
\emph{Sparse adaptive tensor Galerkin approximations of stochastic PDE-constrained control problems},
SIAM/ASA J. Uncert. Quant. 4(1) (2016), 1034-1059. 

\bibitem{KS2013}
A. Kunoth, and C. Schwab,
\emph{Analytic regularity and GPC approximation for control problems constrained by linear parametric elliptic and parabolic PDEs},
SIAM J. Control Optim.  51(3) (2013), 2442-2471. 

\bibitem{KSS2012}
F.Y. Kuo, Ch. Schwab, and I.H. Sloan,
\emph{Quasi-Monte Carlo finite element methods for a class of elliptic partial 
differential equations with random coefficients},
SIAM J. Numer. Anal. 50(6) (2012), 3351-3374.

\bibitem{K1990}
H. J. Kushner,
\emph{Numerical methods for stochastic control problems in continuous time},
SIAM J. Control Optim. 28(5) (1990), 999-1048.

\bibitem{LeHePreWa21} 
C. Lehrenfeld, F. Heimann, J. Preu\ss{  } and H. von Wahl: 
ngsxfem: Add-on to NGSolve for geometrically unfitted finite element discretizations, 
Journal of Open Source Software, 6(64), 3237, https://doi.org/10.21105/joss.03237.

\bibitem{LR2017}
C.\ Lehrenfeld, and A.\ Reusken, \emph{Optimal preconditioners for Nitsche-XFEM discretizations of interface problems}, Numer. Math. 135 (2017), 313-332.

\bibitem{L1971} 
J. L.\ Lions,
\emph{Optimal control of systems governed by partial differential equations}, 
vol. 170, Springer Verlag, 1971.

\bibitem{L19}
A. Lozinski, \emph{CutFEM without cutting the mesh cells: a new way to impose Dirichlet and Neumann boundary contitions on unfitted meshes},
arXiv: 1901.03966, 2019.

\bibitem{LGR2018}
T. Ludescher, S. Gross, and A. Reusken, \emph{A multigrid method for unfitted finite element discretizations of elliptic interface problems}, arXiv: 1807.10196, 2018.

\bibitem{LZ2004}
Z. Lu, and D. Zhang,  
\emph{A  comparative  study  on  quantifying  uncertainty  of  flow  in  randomly heterogeneous media using Monte Carlo simulations, the conventional and KL-based
moment-equation approaches}, SIAM J. Sci. Comput. 26 (2004), 558-577.

\bibitem{MI05}
R. Mittal, and G. Iaccarino, \emph{Immersed boundary methods}, 
Annual Review of Fluid Mechanics 37(1) (2005), 239-261.

\bibitem{MS18}
A. Main, and G. Scovazzi, \emph{The shifted boundary method for embedded domain computations. Part I: Poisson and Stokes problems}, J. Comput. Phys. 372 (2018), 972-995.

\bibitem{MKN2018}
M. Martin, S. Krumscheid and F. Nobile,
\emph{Analysis of stochastic gradient methods for PDE-constrained optimal control problems with uncertain parameters}, MATHICSE Technical Report Nr. 04.2018 March 2018, doi:10.5075/epfl-MATHICSE-263568.

\bibitem{GM2013}
M. D. Gunzburger, and J. Ming,
\emph{Efficient numerical methods for stochastic partial differential equations through transformation to equations driven by correlated noise},
Inter. J. Uncert. Quant. 3 (2013), 321-339. 

\bibitem{N1994}
H. Niederreiter, \emph{Random Number Generation and quasi-Monte Carlo methods}, 
SIAM, Philadelphia, 1994.

\bibitem{NC2006}
D. Nuyens and R. Cools, \emph{Fast algorithms for component-by-component construction of rank-1 lattice rules in shift-invariant reproducing kernel Hilbert spaces}, 
Math. Comput. 75 (2006), 903-920.

\bibitem{PHO16}
V.  Pasquariello, G. Hammerl, F. \''{O}rley, S. Hickel, C. Danowski, A. Popp, W. A. Wall, and N. A. Adams, \emph{A cut-cell finite volume--finite element coupling approach for fluid--structure interaction in compressible flow}, 
J. Comput. Phys. 307 (2016), 670-695.

\bibitem{P1972}
C. S.\ Peskin,
\emph{Flow patterns around heart valves: A numerical method},  
J. Comput. Phys.\ 10 (1972), 252-271.

\bibitem{RDW2010} 
T.\ Rees, H. S.\ Dollar, and A. J.\ Wathen,
\emph{Optimal solvers for PDE-constrained optimization},
SIAM J. Sci. Comput. 32(1) (2010), 271-298.

\bibitem{S2003}
Y.\ Saad, 
\emph{Iterative methods for sparse linear systems}, 
Society for Industrial and Applied Mathematics, Philadelphia, PA, USA, 2nd edition, 2003.

\bibitem{Scho14} 
J. Sch{\"o}berl: 
C++11 Implementation of Finite Elements in NGSolve, 
ASC Report 30/2014, Institute for Analysis and Scientific Computing, Vienna University of Technology, (2014).

\bibitem{tuned_exhaust77} 
P. H. Smith, 
\emph{The Design and Tuning of Competition Engines},  
6th Edition, Cambridge, Mass., R. Bentley, 1977.

\bibitem{TKXP2012}
H. Tiesler, R. M. Kirby, D. Xiu, and T. Preusser, 
\emph{Stochastic collocation for optimal control problems with stochastic PDE constraints}, 
SIAM J. Control Optim. 50(2) (2012), 2659-2682.

\bibitem{Tor2002} 
S. Torquato, 
\emph{Random Heterogeneous Materials: Microstructure and Macroscopic Properties}, Interdisciplinary Applied Mathematics 16, series, Springer-Verlag New York, 2002.

\bibitem{Tr} 
F.\ Tr\"{o}ltzsch, 
\emph{On finite element error estimates for optimal control problems with elliptic PDEs}, 
In Proceedings of the Conference ``Large-Scale Scientific Computations'', 
Springer Lect. Notes in Comp. Sci., 2009.

\bibitem{T2010} 
F. Tr\"{o}ltzsch, 
\emph{Optimal control of partial differential equations: Theory, methods and 
Applications}, vol. 112, American Mathematical Soc., 2010.

\bibitem{WYX2019} 
T. Wang, C. C. Yang, and X. Xie, 
\emph{Extended finite element methods for optimal control problems
governed by Poisson equation in non-convex domains},
Sci. China Math. (2019).

\bibitem{YK1991}
J. Yang and H.J. Kushner, 
\emph{A Monte Carlo method for sensitivity analysis and parametric optimization of nonlinear stochastic systems}, SIAM J. Control Optim. 29(5) (1991), 1216-1249.

\bibitem{YWX2018} 
C. C. Yang, T. Wang, and X. Xie, 
\emph{An interface--unfitted finite element method for elliptic interface optimal control problem}, Numer. Math. Theor. Meth. Appl., 12 (2019),  727-749. 

\bibitem{WFC13}
C. H. Wu, O. M. Faltinsen, B. F. Chen, 
\emph{Time-independent finite difference and ghost cell method to study Sloshing Liquid in 2D and 3D tanks with internal structures}, Communications Comput. Phys. \ 13(3), (2013) 780-800.

\bibitem{ngsxfem} 
ngsxfem - Add-On to NGSolve for unfitted finite element discretizations,
https://github.com/ngsxfem/ngsxfem, 2020.

\bibitem{ngsolve} 
NGSolve - high performance multiphysics finite element software, https://github.com/NGSolve/ngsolve, 2020.
\end{thebibliography}
\end{document}